\documentclass[12pt,twoside]{amsart}
\usepackage[left=3.2cm,right=3.2cm,top=3.2cm,bottom=3.2cm]{geometry}
\usepackage{amssymb}
\usepackage{amsmath}
\usepackage{bbm}
\usepackage[dvips]{graphicx}
\usepackage{hyperref}
\usepackage[capitalize,nameinlink,noabbrev,compress]{cleveref}
\usepackage[dvipsnames]{xcolor}
\usepackage{todonotes}
\usepackage{enumitem}
\usepackage{comment}

\usepackage{tikz}
\usetikzlibrary{arrows}
\usetikzlibrary{decorations.markings}
\usetikzlibrary{shapes.geometric}

\numberwithin{equation}{section}

\newtheorem{thm}{Theorem}[section]
\newtheorem{prop}[thm]{Proposition}

\newtheorem{lem}[thm]{Lemma}

\newtheorem{rem}[thm]{Remark}
\newtheorem{conj}[thm]{Conjecture}

\newmuskip\pFqmuskip
\newcommand*\pFq[6][8]{%
  \begingroup 
  \pFqmuskip=#1mu\relax
  \mathcode`\,=\string"8000
  \begingroup\lccode`\~=`\,
  \lowercase{\endgroup\let~}\pFqcomma
  {}_{#2}F_{#3}{\left[\genfrac..{0pt}{}{#4}{#5};#6\right]}%
  \endgroup
}
\newcommand{\pFqcomma}{\mskip\pFqmuskip}

\newcommand{\ds}{\displaystyle}
\DeclareMathOperator{\sgn}{sgn}
\newcommand{\tsscpp}{totally symmetric self-complementary plane partition}
\newcommand{\pp}{plane partition}

\title[Correlations in TSSCPPs]
{Correlations in totally symmetric self-complementary plane partitions}

\author{Arvind Ayyer}
\address{Arvind Ayyer, Department of Mathematics, 
Indian Institute of Science, Bangalore  560012, India.}
\email{arvind@iisc.ac.in}

\author{Sunil Chhita}
\address{Sunil Chhita, Department of Mathematical Sciences, Durham University, Durham, United Kingdom}
\email{sunil.chhita@durham.ac.uk}


\date{\today}

\begin{document}

\begin{abstract}
Totally symmetric self-complementary plane partitions (TSSCPPs) are boxed plane partitions with the maximum possible symmetry.  
We use the well-known representation of TSSCPPs as a dimer model on a honeycomb graph enclosed in one-twelfth of a hexagon with free boundary
to express them as perfect matchings of a family of non-bipartite planar graphs.
Our main result is that the edges of the TSSCPPs form a Pfaffian point process, for which we give explicit formulas for the inverse Kasteleyn matrix. 
Preliminary analysis of these correlations are then used to give a precise conjecture for the limit shape of TSSCPPs in the scaling limit.

\end{abstract}

\maketitle

\section{Introduction}
\label{sec:intro}

Totally symmetric self-complementary plane partitions (TSSCPPs) of order $n$ are the subset of plane partitions in a $(2n) \times (2n) \times (2n)$ box with the maximum possible symmetry. They have been intensely studied since the initial analysis by Mills, Robbins and Rumsey conjecturing that the number of TSSCPPs of order $n$ are the same as the number of alternating sign matrices (ASMs) of size $n$~\cite{MRR86} (see the formula for $A_n$ in \eqref{asm}). The fact that the number of TSSCPPs of order $n$ is given by $A_n$ was established by Andrews in a 
difficult paper~\cite{Andrews-1994}. 
It was this paper by Mills, Robbins and Rumsey that led Stanley to initiate the study of symmetry classes of plane partitions~\cite{stanley-1986}. 
The fact that ASMs are also enumerated by the same sequence of numbers is known as the Alternating Sign Matrix theorem and was proven first by Zeilberger~\cite{Zei96a} by directly comparing them with TSSCPPs and later by Kuperberg~\cite{Kup96} using a connection with the six-vertex model in statistical mechanics.
See the book by Bressoud~\cite{Bressoud-99} for more about the history of this and related fascinating problems.

Boundary correlations in ASMs have been studied almost from the very beginning. A formula for the enumeration of ASMs according to the position of the unique $1$ in the first row is known as the Refined Alternating Sign Matrix theorem and this was proven by Zeilberger~\cite{Zei96b}. Stroganov~\cite{stroganov-2006} gave a formula for the number of ASMs according to the position of the unique $1$ in the first and last rows (the \emph{top-bottom} formula), and according to the position of the unique $1$ in the first row and left column (the \emph{top-left} formula). The latter formula was subsequently improved in \cite{behrend-dif-zinn-2013}.
Other refined enumeration formulas include the \emph{top two} formula~\cite{fischer-romik-2009,karklinsky-romik-2010},
the \emph{top two and bottom} formula~\cite{fischer-2012}, 
the \emph{top-left-bottom} formula~\cite{fischer-2012,ayyer-romik-2013}
and the \emph{top-left-bottom-right} formula~\cite{ayyer-romik-2013,behrend-2013}. The problem of computing bulk correlations seems like a difficult and interesting open problem.
On the TSSCPP side, no formulas are known for any correlation functions.

On the other hand, correlations for some plane partitions have been established in recent years~\cite{OR01,Pet14,BBNV18}.
The typical perspective here is to view the plane partition as a rhombus or lozenge tiling.  Randomness is introduced by picking each configuration at random from the set of all possible configurations in some prescribed manner, the simplest being picking each configuration uniformly at random which is the case considered here for TSSCPPs.  For a specific class of tiling models, interesting probabilistic features are observed when the system size becomes large, such as a macroscopic limit shape, which is a type of law of large numbers result. Around this limit shape, there are still microscopic fluctuations which are believed to be governed by universal probability distributions originating in both statistical mechanics and random matrix theory. This assertion has been proved  primarily for domino and lozenge tiling models; see~\cite{Gor20} and references therein for details.

To study these fluctuations, one of the more successful approaches has been to study the correlations of an associated particle system to the random tiling model using methods originating from random matrix theory.  For many types of tiling models, these correlations  are governed by the determinant of some matrix, called the correlation kernel.   Probability measures of this form are known as determinantal point processes; see for example~\cite{Sos06}.  Finding the correlation kernel can be computationally tricky, but there are now some relatively standard approaches such as using the Eynard-Mehta Theorem~\cite{Jo03b,RB04} which has been particularly useful for those in the Schur process class~\cite{Jo03, BF08, Dui11, Pet14, DM15} as well as those that are not Schur processes~\cite{DK:17, BD19, CDKL:19}. Put bluntly, this theorem gives the correlation kernel when the model is expressed in terms of nonintersecting lattice paths with fixed endpoints.  There are other approaches for computing correlation kernels, such as vertex operators~\cite{OR01, BCC:14,  BBCCR:15} and also the Harish-Chandra/Itzykson-Zuber integral~\cite{Nov15} 

The Eynard-Mehta theorem has a Pfaffian analog where the final positions of the nonintersecting lattice paths are free. In this case, the correlations of the associated particle system to the tiling model are given by a Pfaffian point process; see~\cite{BBNV18} for an example where the authors give a formula for the correlation kernel for both symmetric plane partitions and plane overpartitions. The TSSCPP is  another example of this and so the Eynard-Mehta theorem immediately shows that the particle system defined through the nonintersecting lattice paths for TSSCPP is a Pfaffian point process with some correlation kernel\footnote{As far as we are aware, this observation is due to Dan Romik, who gave a talk at the Clay Mathematics Institute, Oxford in May 2015 pointing out the difficulty in finding the specific form of the correlation kernel.}. Unfortunately, the formula for this correlation kernel is not known, due to computational difficulties in inverting an arbitrary sized matrix that is found in the Eynard-Mehta theorem.  In this paper, we use dimer model techniques to settle this problem and find a formula for the inverse Kasteleyn matrix for TSSCPPs, where the inverse Kasteleyn matrix can be heuristically thought of as the dimer model equivalent to the correlation kernel of a particle system.   Since lozenge tilings and its associated particle system are in bijection, this implies a formula for the correlation kernel of the associated particle system. 

\begin{center}
\begin{figure}[htbp!]
\includegraphics[scale=1]{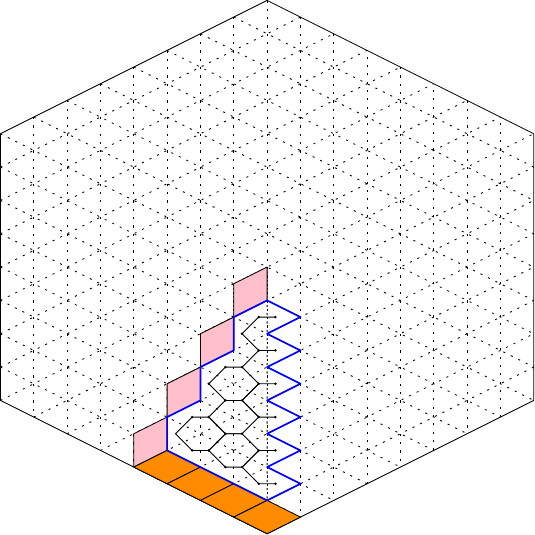}
\caption{The region of the regular hexagon bounded by the blue lines whose tiling is sufficient to determine a \tsscpp{}. The pink and orange lozenges are forced.}
\label{fig:tsscpp-region}
\end{figure}
\end{center}

The rest of the paper is organized as follows. In \cref{sec:graphs}, we convert plane partitions to perfect matchings of a class of non-bipartite graphs bijectively to be able to explain our main results. We then summarize the main results of this article in \cref{sec:summary}, giving a formula for the inverse Kasteleyn matrix and stating a sum rule. \cref{sec:vertexb,sec:bothtop} are devoted to computations of this formula at special locations. \cref{sec:main} completes the proof of the formula for the inverse Kastelen matrix.
The proof of the main result use combinatorial identities whose proof is deferred to \cref{sec:identities}. We present boundary recurrences for the inverse Kasteleyn matrix of independent interest, which we use to prove the sum rule, in \cref{sec:additionrec}. Finally, we end with heuristics for the limit shape and a precise conjecture in \cref{sec:conjlimitshape}.

\section{From TSSCPPs to hexagonal graphs}
\label{sec:graphs}

As mentioned earlier, a \pp{} inside an $a \times b \times c$ box can be equivalently viewed as a lozenge (or rhombus) tiling of a hexagonal region of side lengths $a,b,c,a,b,c$ of a triangular lattice~\cite{kuperberg-1994}. 
A \tsscpp{} of order $n$ is then a rhombus tiling of a regular hexagon with side length $2n$ with the maximum possible symmetry. In this case, all the information about the tiling is contained in $(1/12)$'th of the hexagon~\cite[Section 8]{MRR86}.

This is illustrated in \cref{fig:tsscpp-region}, where the region enclosed by the blue lines is to be tiled with lozenges in a maximal way.
This means that among the $2(n-1)$ pendant edges, only $n-1$ will be matched. This is known as a {\em free boundary condition}.
Equivalently, we have to find maximum matchings of the dual graph drawn in black. 
(Recall that a {\em maximum matching} of a graph is one which has the largest number of matched edges.)

Let $T_{n-1}$ denote this dual graph for TSSCPPs of size $n$. \cref{fig:dualgraphs} shows the dual graphs for TSSCPPs of sizes $3$ and $4$.
Define, for $n \geq 1$,
\begin{equation}
\label{asm}
A_n = \prod_{i=0}^{n-1} \frac{(3i+1)!}{(n+i)!}.
\end{equation}

\begin{thm}[Conjectured in \cite{MRR86}, proved in \cite{Andrews-1994}]
The number of maximum matchings of $T_n$ is given by $A_{n+1}$.
\end{thm}

\begin{figure}[htbp!]
\begin{center}
	\includegraphics[scale=1]{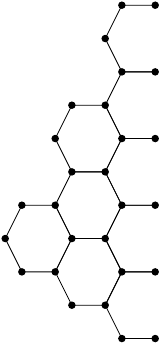}
	\hspace{2cm}
	\includegraphics[scale=0.8]{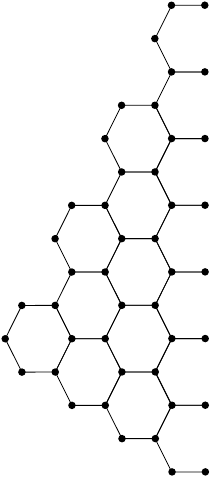}
	\caption{The left figure shows the graph $T_3$ while the right figure shows $T_4$.} 
	\label{fig:dualgraphs}
\end{center}
\end{figure}

Recall that $T_n$ has $2n$ pendant vertices on the right. We now define a related family of graphs $G_{n}$ starting from $T_n$ as follows.
We add $2n+1$ (resp. $2n+2$) vertices if $n$ is even (resp. odd)
in a column on the right of the pendant vertices and connect them in a triangular fashion as illustrated in \cref{fig:hexagons}. Notice that if $n$ is odd, the topmost vertex is a leaf, i.e. a pendant vertex.

\begin{figure}[htbp!]
\begin{center}
	\includegraphics[scale=1]{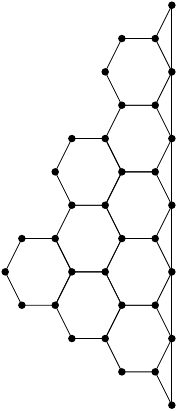}
	\hspace{2cm}
	\includegraphics[scale=0.8]{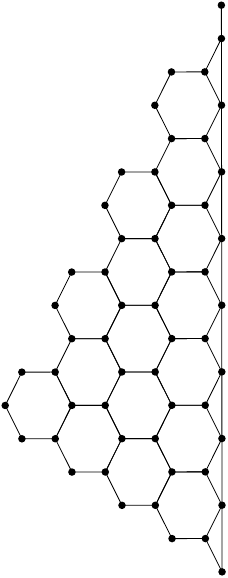}
	\caption{The left figure shows the graph $G_3$ while the right figure shows $G_4$.} 
	\label{fig:hexagons}
\end{center}
\end{figure}

\begin{prop}
The number of perfect matchings of $G_n$ is $A_{n+1}$.
\end{prop}

\begin{proof}
We will construct a bijection between maximum matchings of $T_n$ and perfect matchings of $G_{n}$. Starting with a perfect matching of $G_{n}$ and removing the vertices in the rightmost column, we obtain a maximum matching of $T_n$. To go the other way, we will show that there is a unique way to complete a maximum matching of $T_n$.

Consider a maximum matching as shown in \cref{fig:maxmatch}. We will now match the remaining vertices in $G_{n}$. We start from the leftmost vertex $v_1$ and match it to the leftmost available vertex. We then find the leftmost unmatched vertex and match it to the second unmatched vertex on the left. We continue this way until all vertices are matched.

\begin{center}
\begin{figure}[htbp!]
\includegraphics[scale=0.5]{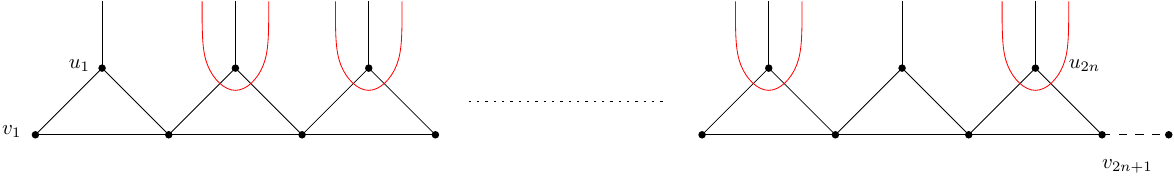}
\caption{An illustration of the pendant vertices of the graph $T_n$ embedded in $G_{n}$. The vertices in the top row are labelled $u_1,\dots$ and those in the bottom row are labelled $v_1,\dots$.
The maximum matching of $T_n$ is shown in red. The rightmost vertex exists only if $n$ is odd.}
\label{fig:maxmatch}
\end{figure}
\end{center}

To see that we do not run into a contradiction, consider the leftmost available vertex at any stage. Suppose it is $v_i$. If $u_i$ is matched, we match $v_i$ to $v_{i+1}$, and if not, we match it to $u_i$. Suppose it is $u_i$. Then we match it to $v_{i+1}$. This is always possible, since  $v_{i+1}$ cannot be matched to $v_i$ as $u_i$ is unmatched. The rightmost vertex has to be matched since there are an even number of unmatched vertices initially.
This completes the proof.
\end{proof}

\section{Summary of results}
\label{sec:summary}

To state our main result, we will need to introduce some notation.
We write $[ x]_2$ to denote $x \mod 2$. For $a,m \in \mathbb{Z}$ with $m>0$, we denote by $\Gamma_{a,a+1,\dots,a+m}$ a positively oriented contour containing the integers $a$ through to $a+m$ and no other integers.  
In particular, $\Gamma_0$ is a positively oriented circle around the origin with radius less than 1.
We denote $\mathrm{i}=\sqrt{-1}$. We denote $Z_{G}$ to be the number of dimer configurations of the graph $G$. For a subset of vertices $U$, we let $Z_{G\backslash U}$ denote the number of dimer configurations on the subgraph of $G$ induced by removing $U$. When $G_n$ is the TSSCPP graph of size $n$ defined above, we write $Z_n \equiv Z_{G_n}$. 
We will also use the notation $Z_n^{\{v_1,\dots,v_m\}}\equiv Z_{G_n\backslash \{v_1, \dots, v_m\}}$, that is the partition function of the dimer model on the {induced graph of $G_n$ on the vertex set $\mathtt{V}_n \backslash \{v_1,\dots,v_m \}$.} Finally, we will use the convention that $\binom{m}{-1}=0$ for all $m \in \mathbb{N} \cup\{0\}$ throughout the paper.

For the graph $G_n = (\mathtt{V}_n, \mathtt{E}_n)$, the coordinates of the vertices are given by 
\begin{equation}	
	\mathtt{V}_n = \{ (x_1,x_2) \mid 0 \leq x_1\leq 2n, x_1 \leq x_2 \leq 2n+1\} \backslash \{ (2n,2n+1) \mathbbm{1}_{n \in 2\mathbb{Z}+1} \} 
\end{equation}
and  the edges are given by 
\begin{equation}
	\begin{split}
		\mathtt{E}_n& =\{ (( x_1,x_2),(x_1+1,x_2)\mid 0\leq x_1 \leq 2n-1,x_1 \leq x_2 \leq 2n+1,[x_1+x_2]_2=1\} \\
 & \cup \{ (( x_1,x_2),(x_1,x_2+1)\mid 0\leq x_1 \leq 2n-\mathbbm{1}_{n \in 2\mathbb{Z}+1},x_1 \leq x_2 \leq 2n\} \\
		&\cup \{ ((x_1,x_1),(x_1+1,x_1+1))\mid 0\leq x_1 \leq 2n-1 \} ;
	\end{split}
\end{equation}
see \cref{fig:rectanglecoords}.

\begin{figure}[htbp!]
\begin{center}
	\includegraphics[height=4.5cm]{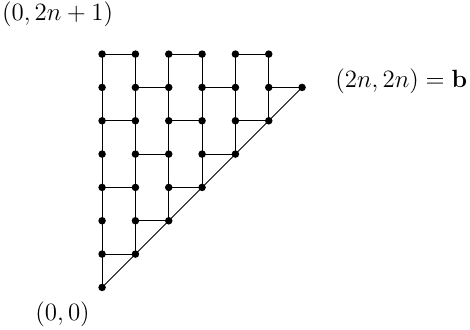}
	\includegraphics[height=4.5cm]{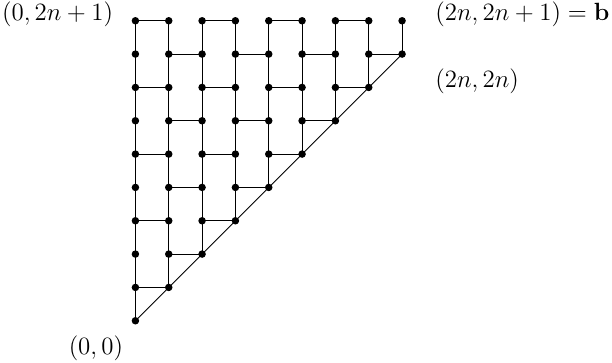}
	\caption{The left figure shows $G_3$, while the right figures shows  $G_4$. The coordinates are for the top left, rightmost and bottom left vertices. The Kasteleyn orientation is chosen so that vertices with odd parity are sinks, those with even parity not on the diagonal are sources, and the edges along the diagonal point towards the origin. } 
	\label{fig:rectanglecoords}
\end{center}
\end{figure}

We let $\mathbf{b}$ denote the vertex $(2n,2n+1-[n]_2)$.
The entries of the skew-symmetric Kasteleyn matrix, $K_n=K_n(x,y)_{x,y \in \mathtt{V}_n}$ are given by
\begin{equation}\label{KasteleynMatrix}
	K_n((x_1,x_2),(y_1,y_2)) =k_n((x_1,x_2),(y_1,y_2))-k_n((y_1,y_2),(x_1,x_2))
\end{equation}
where
\begin{equation}	
	k_n((x_1,x_2),(y_1,y_2)) = \left\{ \begin{array}{ll} 
		1 & \mbox{if } [x_1+x_2]_2 =0 , x_2=y_2, x_1-y_1=1 \\
		1 & \mbox{if }  [x_1+x_2]_2=0,|y_2-x_2|=1, x_1=y_1  \\
		1 & \mbox{if } x_1=x_2, y_1 = y_2 = x_1-1 \\
	0 & \mbox{otherwise.}  \end{array}  \right.  
	\end{equation}
Here the Kasteleyn orientation is chosen so that the number of counter-clockwise arrows around each face is an odd number. From~\cite{Kas61,Kas63}, $|\mathrm{Pf} K_n|$ is equal to $Z_n$.   
	
Define 
	\begin{equation}\label{eq:pnkl}
	p(n,k,\ell) = \frac{ (n+k-2\ell+1)!(2n-k-\ell+1)!}{(k-\ell)! (3n-k+2-2\ell)!} (-1)^k (3n-3k+2) C_{n-k},
\end{equation}
where $C_n = \frac{1}{n+1} \binom{2n}{n}$ is the $n$'th Catalan number. 
We now introduce notation helpful for stating the main theorem.
Recall that $\Gamma_0$ is a positively oriented circle around the origin with radius less than 1.
Introduce the following formulas for $0 \leq i \leq 2n-1$:
	\begin{equation}\label{hnb0}
		h_n^{0,\mathbf{b}}(i) =-[i+1]_2+ \sum_{k=0}^n p(n,k,0) \frac{1}{2 \pi \mathrm{i}} \int_{\Gamma_0} \text{d}r \frac{(1+r)^{n-k}}{(1-r)r^{i-2k}}, 
	\end{equation}
and
\begin{equation}\label{hnb1}
	h_n^{1,\mathbf{b}}(i) = \sum_{k=0}^n p(n,k,0) \frac{1}{2 \pi \mathrm{i}} \int_{\Gamma_0} \text{d}r \frac{(1+r)^{n-k}}{r^{i-2k+1}}.
	\end{equation}
We will first state the formula for $K_n^{-1}$ in the special case when the second vertex is $\mathbf{b}$. This will prove useful for the general result.

\begin{thm}
\label{thm:vertexb}
If $(x_1,x_2)=(i_1,i_1+2i_2+1)$ with $ 0\leq i_2 \leq n - \lceil i_1/2 \rceil$ then 
	\begin{equation} \label{Knb1}
		K_n^{-1}((x_1,x_2), \mathbf{b}) = \sum_{l=0}^{i_1} (-1)^{i_2+l} \binom{i_2-1+l}{l} h_n^{1,\mathbf{b}}(i_1-l),
	\end{equation}
 and if $(x_1,x_2)=(i_1,i_1+2i_2)$ with $ 0\leq i_2 \leq n - \lfloor i_1/2 \rfloor$ then 
\begin{equation}\label{Knb0}
K_n^{-1}((x_1,x_2), \mathbf{b}) = \sum_{l=0}^{i_2} (-1)^{i_2} \binom{ i_2}{ l}  h_n^{0,\mathbf{b}}(i_1+l).
	\end{equation}
	Finally, when $n$ is even, $K^{-1}_n((2n,2n),\mathbf{b})=-1$.
\end{thm}

Introduce the following formulas for $0 \leq i,j \leq 2n-1$:
\begin{equation}\label{tn00}
	\begin{split}
		&t_n^{0,0} (i,j)= \mathbbm{1}_{[i<j]}[i+1]_2[j]_2 - \mathbbm{1}_{[i>j]}[i]_2[j+1]_2 +[j+1]_2h_n^{0,\mathbf{b}}(i)-[i+1]_2 h_n^{0,\mathbf{b}}(j) \\
	 & +	\sum_{k_1=0}^n \sum_{\ell_1=0}^{k_1} \sum_{k_2=0}^n \sum_{\ell_2=0}^{k_2} p(n,k_1,\ell_1)p(n,k_2,\ell_2) \frac{1}{(2\pi \mathrm{i})^4} \int_{\Gamma_0}\text{d}r_1 \int_{\Gamma_0} \text{d}r_2 \int_{\Gamma_0}\text{d}s_1 \int_{\Gamma_0}\text{d}s_2 \\
&  \times \frac{ (1+r_1)^{n-k_1} (1+r_2)^{n-k_2}}{(1-r_1)r_1^{i-2k_1}(1-r_2)r_2^{j-2k_2}}
		\frac{s_1-s_2}{(s_1 s_2-1)s_1^{\ell_1+1}s_2^{\ell_2+1}}, 
\end{split}
\end{equation}
\begin{equation}\label{tn10}
\begin{split}
	&t_n^{1,0} (i,j)= [j+1]_2h_n^{1,\mathbf{b}}(i)\\
& +	\sum_{k_1=0}^n \sum_{\ell_1=0}^{k_1} \sum_{k_2=0}^n \sum_{\ell_2=0}^{k_2} p(n,k_1,\ell_1)p(n,k_2,\ell_2) \frac{1}{(2\pi \mathrm{i})^4} \int_{\Gamma_0}\text{d}r_1 \int_{\Gamma_0} \text{d}r_2 \int_{\Gamma_0}\text{d}s_1 \int_{\Gamma_0}\text{d}s_2\\
	& \times \frac{ (1+r_1)^{n-k_1} (1+r_2)^{n-k_2}}{r_1^{i-2k_1+1}(1-r_2)r_2^{j-2k_2}}	\frac{s_1-s_2}{(s_1 s_2-1)s_1^{\ell_1+1}s_2^{\ell_2+1}},
\end{split}
\end{equation}
\begin{equation}\label{tn01}
\begin{split}
	&t_n^{0,1} (i,j)= -[i+1]_2 h_n^{1,\mathbf{b}}(j)\\
	& +	\sum_{k_1=0}^n \sum_{\ell_1=0}^{k_1} \sum_{k_2=0}^n \sum_{\ell_2=0}^{k_2} p(n,k_1,\ell_1)p(n,k_2,\ell_2) \frac{1}{(2\pi \mathrm{i})^4} \int_{\Gamma_0}\text{d}r_1 \int_{\Gamma_0} \text{d}r_2 \int_{\Gamma_0}\text{d}s_1 \int_{\Gamma_0}\text{d}s_2\\
& \times  \frac{ (1+r_1)^{n-k_1} (1+r_2)^{n-k_2}}{(1-r_1)r_1^{i-2k_1}r_2^{j-2k_2+1}} 	\frac{s_1-s_2}{(s_1 s_2-1)s_1^{\ell_1+1}s_2^{\ell_2+1}},
\end{split}
\end{equation}
and
\begin{equation}\label{tn11}
\begin{split}
	&t_n^{1,1} (i,j)=
	\sum_{k_1=0}^n \sum_{\ell_1=0}^{k_1} \sum_{k_2=0}^n \sum_{\ell_2=0}^{k_2} p(n,k_1,\ell_1)p(n,k_2,\ell_2) \frac{1}{(2\pi \mathrm{i})^4} \int_{\Gamma_0}\text{d}r_1 \int_{\Gamma_0} \text{d}r_2 \int_{\Gamma_0}\text{d}s_1 \int_{\Gamma_0}\text{d}s_2\\
& \times \frac{ (1+r_1)^{n-k_1} (1+r_2)^{n-k_2}}{r_1^{i-2k_1+1}r_2^{j-2k_2+1}}  	\frac{s_1-s_2}{(s_1 s_2-1)s_1^{\ell_1+1}s_2^{\ell_2+1}}.
\end{split}
\end{equation}

We are now ready to state the formula for $K_n^{-1}$. Below, we make repeated use of $K^{-1}_n(x,y)=-K^{-1}_n(y,x)$ for $x,y \in \mathtt{V}_n$ since $K_n$ is antisymmetric.

\begin{thm}
\label{thm:main}
Suppose that $(x_1,x_2)=(i_1,i_1+2i_2+ \epsilon_i)$, $(y_1,y_2)=(j_1,j_1+2j_2+\epsilon_j)$ with $0\leq i_1 ,j_1 \leq  2n-1$
and $\epsilon_i, \epsilon_j \in \{0,1\}$, where $0 \leq i_2 \leq n- \lfloor (i_1 + \epsilon_i)/2 \rfloor$, and $0 \leq j_2 \leq n- \lfloor (j_1 + \epsilon_j)/2 \rfloor $. 
	
\begin{itemize}
	
\item If $\epsilon_i = \epsilon_j= 1$, then
	\begin{equation}\label{Kinv11}
		\begin{split}
			K_n^{-1}((x_1,x_2),(y_1,y_2)) &= \sum_{\ell_1=0}^{i_1} \sum_{\ell_2=0}^{j_1}  (-1)^{i_2+j_2} (-1)^{\ell_1+\ell_2} \\ & \times \binom{i_2-1+\ell_1}{\ell_1}\binom{j_2-1+\ell_2}{\ell_2}
 t_n^{1,1}(i_1-\ell_1,j_1-\ell_2).
		\end{split}
	\end{equation}

\item If $\epsilon_i = \epsilon_j= 0$, then
	\begin{equation}\label{Kinv00}
		\begin{split}
			K_n^{-1}((x_1,x_2),(y_1,y_2)) &= \sum_{\ell_1=0}^{i_2} \sum_{\ell_2=0}^{j_2}  (-1)^{i_2+j_2}  \binom{i_2}{\ell_1}\binom{ j_2}{\ell_2} t_n^{0,0}(i_1+\ell_1,j_1+\ell_2).
		\end{split}
	\end{equation}

\item If $\epsilon_i = 1$ and $\epsilon_j= 0$, then
	\begin{equation}\label{Kinv10}
		\begin{split}
			&K_n^{-1}((x_1,x_2),(y_1,y_2)) = \sum_{\ell_1=0}^{i_1} \sum_{\ell_2=0}^{j_2}  (-1)^{i_2+j_2} (-1)^{\ell_1} \binom{ i_2-1+\ell_1}{ \ell_1}\binom{j_2}{\ell_2}\\ & \times  t_n^{1,0}(i_1-\ell_1,j_1+\ell_2) - (-1)^{i_2+j_2} \mathbbm{1}_{x_1 \geq y_1} \mathbbm{1}_{x_1+x_2<y_1+y_2}  \binom{j_2-i_2-1}{ i_1-j_1}.
		\end{split}
	\end{equation}

\end{itemize}
\end{thm}

Note that the case of $\epsilon_i =0$ and $\epsilon_j = 1$ in \cref{thm:main} is settled by the antisymmetry of $K_n$.
\cref{thm:vertexb} and \cref{thm:main} could be proven in principle by verifying the matrix equation $K_n.K_n^{-1}=\mathbb{I}$ directly. 
However, this approach seems computationally very difficult.

Our proof strategy is as follows. We first establish \cref{thm:vertexb} in \cref{sec:vertexb}.
  Once this is established, we use graphical condensation~\cite{Kuo06} to establish $K^{-1}_n(x,y)$ for $x$ and $y$ on the top boundary of the TSSCPP.  We can then recover the other entries using recursions obtained from the matrix equations $K_n^{-1}.K_n=K_n.K_n^{-1}=\mathbb{I}$ viewed entrywise. 
The proof is given in \cref{sec:main}, while postponed proofs of results used in proving \cref{thm:main} are given in \cref{sec:identities}.

This leads to the question whether there is a set of recurrences that establishes $K_n^{-1}$ uniquely, without relying on a priori expressions for $K_n^{-1}(\cdot,\mathbf{b})$. Such recurrences are known for domino tilings of the Aztec diamond~\cite{CY14} and are expected to hold for lozenge tilings~\cite{Oko09}. In fact, there is an additional recurrence, along with the recurrence from graphical condensation given in \cref{lem:condensation} which parametrizes $K_n^{-1}$ on the boundary. These two recurrences give a concrete and explicit example of the additional recurrences postulated in \cite{Oko09} for lozenge tilings. These recurrences combined with the recurrences from the matrix equations $K_n.K_n^{-1}=K_n^{-1}.K_n=\mathbbm{I}$ give a unique way to determine $K_n^{-1}$.  This additional recurrence is given in \cref{sec:additionrec}. 

One consequence of finding $K_n^{-1}$ is a formula to compute local statistics~\cite{MPW63,Ken97}.  We restate this here in a form that applies to our situation. 

\begin{thm}[{Montroll-Potts-Ward~\cite{MPW63}, Kenyon~\cite{Ken97}}]
\label{thm:localstats}
	The probability of edges $e_1=(v_1,v_2),\dots, e_m=(v_{2m-1},v_{2m})$ on $\mathtt{V}_n$ is given by 
	\begin{equation}
		\mathbb{P}[e_1,\dots, e_n] = \prod_{k=1}^m K_n(v_{2k-1},v_{2k}) \; \mathrm{Pf} \left( K^{-1}_n(v_i,v_j) \right)^T_{1 \leq i,j \leq 2m}.
	\end{equation}
\end{thm}

\begin{figure}[htbp!]
\begin{center}
	\includegraphics[width=\textwidth]{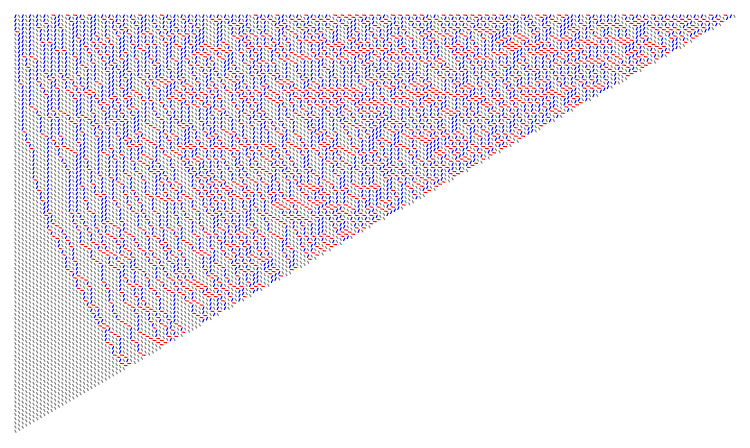}
	\caption{A simulation, using Glauber dynamics, of a uniformly random TSSCPP of size 100. Here, we have rotated the hexagonal graph in \cref{fig:hexagons} by $\pi/6$. The simulations only show the dimers on this graph, drawn in different three different colors.   } 
	\label{fig:simulation}
\end{center}
\end{figure}

The term ``sum rule'' refers to certain sums of correlation functions that are of importance. In many cases, these sum rules have simple formulas. For the so-called quantum Knizhnik-Zamolodchikov equation, sum rules have been conjectured relating them to the $q$-enumeration of TSSCPPs by Di Francesco~\cite{dif-2006}, and Di Francesco and Zinn-Justin~\cite{ZinDif05,ZinDif08} and proved by 
Zeilberger~\cite{zeilberger-2007} .

We now state a formula for a sum rule for TSSCPPs. For $0 \leq j \leq 2n-1$, let 
\begin{equation}
\label{eq:Kinvboundarycount}
	g_n^\mathbf{b}(j)=|K^{-1}_n((j,j),\mathbf{b})|=(-1)^{j+1}K^{-1}_n((j,j),\mathbf{b}) =\frac{Z_n^{\{(j,j),\mathbf{b}\}}}{Z_n}.
\end{equation}

\begin{thm}
\label{thm:gnb formula}
\[
\sum_{j=0}^{2n} g_n^{\mathbf{b}} (j) = 
\begin{cases}
n+1 & \text{$n$ even}, \\
n+ \frac 12 & \text{$n$ odd}.
\end{cases}
\]
\end{thm}

We remark that the entries $g_n^{\mathbf{b}} (j)$ themselves are not so simple. For example,
\[
(g_2^{\mathbf{b}} (j))_{j=0}^4 = \left(1, \frac{2}{7}, \frac{1}{7},
\frac{4}{7}, 1 \right),
\]
and
\[
(g_3^{\mathbf{b}} (j))_{j=0}^6 = \left(1, \frac{1}{6}, \frac{1}{3},
\frac{11}{14}, \frac{17}{21}, \frac{17}{42}, 0 \right).
\]

The formula given in \cref{thm:main} is not in the most convenient form for asymptotic analysis.  Nonrigorous computations show that the terms in the formula can be approximated by double contour integral formulas, that share some similarities with lozenge tiling models in the Schur class, e.g. see~\cite{GP19} and references therein. Furthermore, these nonrigorous computations show a limit shape and Airy kernel statistics at the edge, which also appears in simulations; see \cref{fig:simulation}.  We will not attempt to pursue these here as the computations appear to be long and involved. However, we will give a conjecture of the limit shape, based on a short computation in \cref{sec:conjlimitshape}.

\begin{conj}\label{conj:limitshape}
Rescale the TSSCPP of size $n$ so that the three corners are given by $(-2,0)$, $(0,0)$ and $(-2,-2/\sqrt{3})$ as $n\to \infty$, that is, for $(x_1,x_2) \in \mathtt{V}_n$, rescale so that $x_1=[(X+2)n]$ and $x_2=[(\sqrt{3}Y+2)n]$. Then, the limit shape curve is given by 
	\begin{equation}
X^2+Y^2=4
	\end{equation}
so that the region $X^2+Y^2 \geq 4$ in the rescaled TSSCPP is frozen. 
\end{conj}

\section{Proof of \cref{thm:vertexb}} 
\label{sec:vertexb}

Before giving the proof of \cref{thm:vertexb}, we give two combinatorial identities which will be used below.   Their proofs are given in \cref{sec:identities}.  

\begin{thm}
\label{thm:sumf}
Let $n \in \mathbb{N}$. Let
\[
f_{n}(k) = p(n,k,0)2^{n-k} = \frac{ (n+k+1)!(2n-k+1)!}{k! (3n-k+2)!} (-1)^k 2^{n-k} (3n-3k+2) C_{n-k} .
\]
Then
\[
\sum_{k=0}^n f_{n}(k) = \frac{1}{2} (1 + (-1)^n) 
= \begin{cases}
1 & \text{if $n$ is even}, \\
0 & \text{if $n$ is odd}. 
\end{cases}
\]
\end{thm}

C. Krattenthaler has kindly shown us how to derive \cref{thm:sumf} starting from an identity in the book by Gasper and Rahman~\cite[Equation (3.8.12)]{gasper-rahman-2004}. However, we use Zeilberger's algorithm to prove it in \cref{sec:identities}.

\begin{thm}
\label{thm:sumg}
Let $n \in \mathbb{N}$ and $0 \leq i \leq n$. Define
\[
g_{n,i}(k,j) =  \frac{1}{2^{n-k}} f_n(k) \binom{2n+1-i-k}{j}.
\]
Then, we have
\begin{equation}
\label{id1}
G_{n,i} = \sum_{k=0}^n \sum_{j=0}^{n+i-2k} g_{n,i}(k,j) = 2^{n-i},
\end{equation}
and
\begin{equation}
\label{id2}
G'_{n,i} = \sum_{k=0}^n \sum_{j=0}^{n+k-2i} g_{n,i}(k,j) = (-1)^n 2^{n-i}. 
\end{equation}
\end{thm}

The proofs of \cref{thm:sumf},\cref{thm:sumg} are given in \cref{sec:identities}. 
First, we need the following two results.

\begin{prop}\label{prop:Kbeven}
When $n$ is even, we have $K^{-1}((2n,2n),\mathbf{b})=-1$.
\end{prop}

\begin{proof}
	Notice that for $n$ even, $K_n((2n,2n),\mathbf{b})=1$. By \cref{thm:localstats}, the probability of observing the edge $((2n,2n),\mathbf{b})$ is equal to $$K_n((2n,2n),\mathbf{b}) K_n^{-1}(\mathbf{b},(2n,2n))=-K_n^{-1}((2n,2n),\mathbf{b}).$$ Since the edge $((2n,2n),\mathbf{b})$ is always observed when $n$ is even, the result follows.
\end{proof}

\begin{prop} 
From the formulas for $K_n^{-1}(\cdot, \mathbf{b})$ given in~\eqref{Knb1} and~\eqref{Knb0}, we have that
\label{prop:claim}
	\begin{equation}
		\begin{split}
			K_n^{-1}((2n-2,2n-2),\mathbf{b})+K_n^{-1}((2n-2,2n-1),\mathbf{b})& +K_n^{-1}((2n-1,2n),\mathbf{b})\\
&=\begin{cases}
1 & \text{if $n$ is even}, \\
0 & \text{if $n$ is odd}. 
\end{cases}
		\end{split}
	\end{equation}
\end{prop}

\begin{proof}
	We expand out the left side using the definitions of $K_n^{-1}(\cdot, \mathbf{b})$ as given in~\eqref{Knb1} and~\eqref{Knb0} which involves the formulas~\eqref{hnb1} and~\eqref{hnb0}. This gives
	\begin{equation}
		\begin{split}
			&\sum_{k=0}^n p(n,k,0) \frac{1}{2 \pi \mathrm{i}} \int_{\Gamma_0} \text{d}r 
\left(\frac{(1+r)^{n-k} }{(1-r)r^{2n-2-2k}} +\frac{(1+r)^{n-k}}{r^{(2n-2)+1-2k}} +\frac{(1-r)^{n-k}}{r^{2n-1+1-2k}} \right)\\
			&=\sum_{k=0}^n p(n,k,0) \frac{1}{2 \pi \mathrm{i}} \int_{\Gamma_0} \text{d}r \frac{(1+r)^{n-k} }{(1-r)r^{2n-2k}} =\sum_{k=0}^n p(n,k,0)2^{n-k},
		\end{split}
	\end{equation}
	where the first equality follows from simplifying the integrand and
	the last equality follows from pushing the contour through infinity and computing the residue at $r=1$. 
	The claim follows from \cref{thm:sumf}.
\end{proof}

We will also repeatedly use the standard integrals,
\begin{equation}
\label{integral1}
	\frac{1}{2\pi \mathrm{i}}\int_{\Gamma_0} \frac{\text{d}r}{r} \frac{(1+r)^b}{r^a} = \binom{b}{a},
\end{equation}
\begin{equation}
\label{integral2}
\frac{1}{2\pi \mathrm{i}}\int_{\Gamma_0} \frac{\text{d}r}{r} \frac{(1+r)^b}{(1-r)r^a} = \sum_{j=0}^a \binom{b}{j},
\end{equation}
\begin{equation}\label{integral3}
	\int_0^\infty \text{d}r \; r^a e^{-r}=a!,
\end{equation}
and
\begin{equation}
	\label{integral4}
	\frac{1}{2\pi \mathrm{i}}\int_{\Gamma_0} \frac{\text{d}r}{r} \frac{e^r}{r^a}= \frac{1}{a!}
\end{equation}
for positive integers $a$ and $b$.

\begin{proof}[Proof of \cref{thm:vertexb}]
We will prove the following, assuming \eqref{Knb1} and \eqref{Knb0}.
	\begin{equation}\label{Kinvboundary}
		K_n.K_n^{-1}(x,\mathbf{b})=\mathbbm{I}_{x=\mathbf{b}}
	\end{equation}
The determinant of $K_n$ is nonzero because it is the square of the partition function. Therefore the equations from~\eqref{Kinvboundary} are linearly independent and they guarantee the uniqueness of
 $K_n^{-1}(\cdot,\mathbf{b})$.
For the purposes of the proof, we introduce
	\begin{equation}
		E_n=\{ (x_1,x_2) \in \mathtt{V}_n \backslash \{\mathbf{b}\}: [x_1+x_2]_2 =0 \},
	\end{equation}
	\begin{equation}
O_n=\{ (x_1,x_2) \in \mathtt{V}_n \backslash \{\mathbf{b}\}: [x_1+x_2]_2 =1 \},
	\end{equation}
and
	\begin{equation}
		D_n=\{ (x_1,x_2) \in E_n: x_1=x_2 \}.
	\end{equation}
We will verify ~\eqref{Kinvboundary} in each of the four cases: $x \in E_n\backslash D_n$, $x \in D_n$, $x = \textbf{b}$ and finally $x \in O_n$.

\noindent
\textbf{Case (i)}:	We first consider $x=(x_1,x_2) \in E_n \backslash D_n$. Expanding out the left side of~\eqref{Kinvboundary} entrywise gives
	\begin{equation} \label{Kinvboundaryproof1}
		K^{-1}_n (x+(-1,0), \mathbf{b})\mathbbm{1}_{x_1>0}  + K^{-1}_n (x+(0,1),\mathbf{b}) \mathbbm{1}_{x_2<2n+1} +K^{-1}_n(x+(0,-1),\mathbf{b}). 
	\end{equation}
	Write $x=(i_1,i_1+2i_2)$ and suppose that $x_1 \not = 0$ and $x_2 \not = 2n+1$. Then~\eqref{Kinvboundaryproof1} becomes by~\eqref{Knb1} 
\begin{equation} \label{Kinvboundaryprooftoprelation}
\begin{split}
&(-1)^{i_2} \Bigg( \sum_{\ell=0}^{i_1-1}(-1)^\ell  \binom{ i_2-1+\ell}{ \ell} h_n^{1,\mathbf{b}}(i_1-\ell-1) \\ 
&	+\sum_{\ell=0}^{i_1}(-1)^\ell \binom{ i_2-1+\ell}{ \ell} h_n^{1,\mathbf{b}}(i_1-\ell) 
-\sum_{\ell=0}^{i_1}(-1)^\ell \binom{ i_2-2+\ell}{ \ell}  h_n^{1,\mathbf{b}}(i_1-\ell) \Bigg) \\
=& (-1)^{i_2}  \Bigg( \sum_{\ell=0}^{i_1-1}(-1)^\ell 	\binom{ i_2-1+\ell}{ \ell}  h_n^{1,\mathbf{b}}(i_1-\ell-1) \\
& +\sum_{\ell=0}^{i_1}(-1)^\ell\binom{ i_2-2+l}{ \ell-1}  h_n^{1,\mathbf{b}}(i_1-\ell) \Bigg)=0,
\end{split}
\end{equation}
where we have used $ \binom{n}{k}+ \binom{n}{k-1}= \binom{n+1}{k}$
in the last step.
This verifies~\eqref{Kinvboundary} when $x=(x_1,x_2) \in E_n \backslash D_n$ and $x_1 \not =0$ and $x_2 \not =2n+1$.
	We now consider~\eqref{Kinvboundaryproof1} when $x_1=0$. Notice that $K^{-1}_n ((0,2i_2+1),\mathbf{b}) = h_n^{1,\mathbf{b}}(0)(-1)^{i_2}$ and so by~\eqref{Knb1} we have also verified~~\eqref{Kinvboundary} for $(0,x_2) \in E_n \backslash D_n$.  Next we consider when $x_2=2n+1$,  which means that~\eqref{Kinvboundaryproof1} becomes 
\begin{equation} \label{Kinvboundaryproof1topright}
K^{-1}_n (x+(-1,0), \mathbf{b})\mathbbm{1}_{x_1>0}  +K^{-1}_n(x+(0,-1),\mathbf{b}). 
\end{equation}
Take $x+(-1,0)=(x_1-1,2n+1)=(i_1,2n+1)$ so that $i_1=x_1-1$. This means that we have
$$
x+(-1,0)=\big(i_1,i_1+2\frac{2n-i_1}{2}+1 \big)
$$
and
$$
x+(0,-1)=\big(i_1+1,i_1+1+2\frac{2n-i_1-2}{2}+1 \big).
$$
Using~\eqref{Knb1}, we have that~\eqref{Kinvboundaryproof1topright} becomes
\begin{equation}
\begin{split} \label{Kinvboundaryproof1toprightrearrange}
&\sum_{\ell=0}^{i_1} (-1)^{n-\frac{i_1}{2}+\ell} \binom{n-\frac{i_1}{2}-1+\ell}{\ell} h_n^{1,\mathbf{b}}(i_1-\ell)\\
&   +\sum_{\ell=0}^{i_1+1} (-1)^{n-\frac{i_1}{2}-1+\ell} \binom{n-\frac{i_1}{2}-2+\ell}{\ell} h_n^{1,\mathbf{b}}(i_1+1-\ell)\\
&=-\sum_{\ell=0}^{i_1+1} (-1)^{n-i_1/2+\ell} \binom{n-\frac{i_1}{2}-1+\ell}{\ell} h_n^{1,\mathbf{b}}(i_1+1-\ell),
\end{split}
\end{equation}
where the last equation is found from rearranging the relation in~\eqref{Kinvboundaryprooftoprelation}. We need to show that the right side of the above equation is equal to zero.  
To do so, we expand out the above term using~\eqref{hnb1} which gives
	\begin{equation} \label{Kinvboundaryproofeq2}
		(-1)^{n -i_1/2}\sum_{k=0}^n p(n,k,0) \frac{1}{2\pi \mathrm{i}} \int_{\Gamma_0} \frac{\text{d}r}{r} \frac{(1+r)^{n-k}r^{2k}}{r^{i_1+1}} \sum_{\ell=0}^{i_1+1} (-r)^\ell \binom{n-i_1/2+\ell-1}{\ell} .
	\end{equation}
The inner sum in the integrand equals 
	\begin{equation}
		\begin{split}
			&\sum_{\ell=0}^{i_1+1} (-r)^\ell \binom{n-i_1/2+\ell-1}{\ell} = (1+r)^{\frac{1}{2} (i_1-2n)} \\&+(-1)^{1+i_1} r^{2+i_1} \binom{n+i_1/2+1}{2+i} \pFq{2}{1}{1,2+i_1/2+n}{3+i}{-r},
		\end{split}
	\end{equation}
	where the ${}_2F_1$ terms is the Gaussian hypergeometric function. {Since the hypergeometric term is non-singular in $r$}, we see that~\eqref{Kinvboundaryproofeq2} is equal to 
	\begin{equation}
		(-1)^{n -i_1/2}\sum_{k=0}^n p(n,k,0) \frac{1}{2\pi \mathrm{i}} \int_{\Gamma_0} \frac{\text{d}r}{r} (1+r)^{i_1/2-k} r^{2k-(i_1+1)}.
	\end{equation}
By \eqref{integral1}, the integral above is equal to $\binom{i_1/2-k}{i_1+1-2k} = 0$. We have thus verified~\eqref{Kinvboundary} for $x \in E_n \backslash D_n$.  

\noindent
\textbf{Case (ii)}:	Next we consider $x=(x_1,x_2) \in D_n$ provided that $x_1 <2n-1$. For this case, we have that the left side of~\eqref{Kinvboundary} entrywise is equal to
	\begin{equation}\label{Kinvboundarydiagonal}
		\big(K^{-1}_n(x-(1,1),\mathbf{b})  + K^{-1}_n(x-(1,0),\mathbf{b})\big)\mathbbm{1}_{x_1>0}+K^{-1}_n(x+(0,1),\mathbf{b})-K^{-1}_n(x+(1,1),\mathbf{b}).
	\end{equation}
	When $0<x_1<2n-1$, the integral contribution in~\eqref{Kinvboundarydiagonal} from using~\eqref{hnb0} and~\eqref{hnb1} is
	\begin{equation}
		\begin{split}
			&\frac{1}{2\pi \mathrm{i}} \int_{\Gamma_0} \text{d}r (1+r)^{n-k} r^{2k} \left( \frac{r^{-x_1+1}}{1-r} +r^{-x_1} +r^{-x_1-1} - \frac{r^{-x_1-1}}{1-r} \right) \\
			&=\frac{1}{2\pi \mathrm{i}} \int_{\Gamma_0} \text{d}r \frac{(1+r)^{n-k}}{1-r} r^{-x_1-1+2k} \left(r^2+(1-r)r+(1-r)-1 \right) =0,
		\end{split}
	\end{equation}
	which verifies~\eqref{Kinvboundary} when $x=(x_1,x_2) \in D_n$ for $0 < x_1 <2n$. When $x_1=0$, we have that~\eqref{Kinvboundarydiagonal} becomes
	\begin{equation}
		K^{-1}_n((0,1),\mathbf{b}) -K^{-1}_n ((1,1),\mathbf{b}) =0
	\end{equation}
by explicit evaluation (both terms have a contribution from $k=0$ in the sum in~\eqref{hnb0} and~\eqref{hnb1}).  When $x=(x_1,x_2) \in D_n$ and $x_1=2n-1$, the left side of~\eqref{Kinvboundary} entrywise is equal to 
\begin{equation}\label{Kinvboundarydiagonalodd}
	K_n^{-1}((2n-2,2n-2),\mathbf{b})+K_n^{-1}((2n-2,2n-1),\mathbf{b})+K_n^{-1}((2n-1,2n),\mathbf{b})- K_n^{-1}(\mathbf{b},\mathbf{b})
\end{equation}
if $n$ is odd, and 
\begin{multline}
\label{Kinvboundarydiagonaleven}
	K_n^{-1}((2n-2,2n-2),\mathbf{b})+K_n^{-1}((2n-2,2n-1),\mathbf{b}) \\
	+K_n^{-1}((2n-1,2n),\mathbf{b})- K_n^{-1}((2n,2n),\mathbf{b})
\end{multline}
if $n$ is even.   
From \cref{prop:claim} and \cref{prop:Kbeven}, we have that both~\eqref{Kinvboundarydiagonalodd} and~\eqref{Kinvboundarydiagonaleven} are equal to zero. 

\noindent
\textbf{Case (iii)}: Next we consider the left side of~\eqref{Kinvboundary} entrywise when $x=\mathbf{b}$ when $n$ is odd (the even case is immediate from \cref{prop:Kbeven}). Expanding out these terms using~\eqref{Knb1},~\eqref{Knb0},~\eqref{hnb1}, and~\eqref{hnb0} gives
\begin{equation}
	\begin{split}
		&K_n^{-1} ((2n-1,2n),\mathbf{b}) +K_n^{-1} ((2n-1,2n-1),\mathbf{b})\\
		&= \sum_{k=0}^n p(n,k,0) \frac{1}{2 \pi \mathrm{i}} \int_{\Gamma_0} \text{d}r \frac{(1+r)^{n-k}}{r^{2n-2k}} \left( 1+ \frac{r}{1-r} \right) \\
		&= \sum_{k=0}^n p(n,k,0) \frac{1}{2 \pi \mathrm{i}} \int_{\Gamma_0} \text{d}r \frac{(1+r)^{n-k}}{r^{2n-2k}}  \frac{1}{1-r}  \\ 
		&= \sum_{k=0}^{n-1} p(n,k,0) \frac{1}{2 \pi \mathrm{i}} \int_{\Gamma_0} \text{d}r \frac{(1+r)^{n-k}}{r^{2n-2k}}  \frac{1}{1-r}  \\ 
		&=\sum_{k=0}^{n-1} p(n,k,0)2^{n-k}=\frac{1}{2}(1+(-1)^n)-(-1)^n 
		=\left\{\begin{array}{ll}
			1 & \mbox{if $n$ odd}\\
		0 & \mbox{if $n$ even}\end{array} \right. 
	\end{split}
\end{equation}
where the third equality follows from the absence of a residue at $r=0$ when $k=n$, the fourth equality follows from \eqref{integral1},
and we have used \cref{thm:sumf} for the fifth equality.
We also need the left side of~\eqref{Kinvboundary} entrywise when $x=(2n,2n)$ which is given by
\begin{equation}
	K_n^{-1} ((2n-1,2n),\mathbf{b}) +K_n^{-1} ((2n-1,2n-1),\mathbf{b})-K_n^{-1} (\mathbf{b},\mathbf{b})=0,
\end{equation}
where the sum of the first two terms follows from the proceeding computation and we have used the antisymmetry of $K_n^{-1}$.  We have now verified~\eqref{Kinvboundary} for $x\in E_n \cup \{\mathbf{b} \}$.  

\noindent
\textbf{Case (iv)}:	Finally, we verify~\eqref{Kinvboundary} for $x \in O_n$. When $x=(x_1,x_2) \in O_n$, we have that the left side of~\eqref{Kinvboundary} entrywise equals 
\begin{equation}\label{Kinvboundaryodd}
	-K^{-1}_n (x+(0,1),\mathbf{b})\mathbbm{1}_{x_2<2n+1}-K^{-1}_n (x+(0,-1),\mathbf{b})-K^{-1}_n (x+(1,0),\mathbf{b}).
\end{equation}
If $x_2<2n+1$, using~\eqref{Knb0} and writing $x=(x_1,x_2)=(i_1,i_1+2i_2+1)$, we have that~\eqref{Kinvboundaryodd} becomes
\begin{equation}\label{Kinvboundaryoddrelation}
	\begin{split}
		&-\sum_{\ell=0}^{i_2+1} (-1)^{i_2+1} \binom{i_2+1}{\ell} h_n^{0,\mathbf{b}} (i_1+\ell) - \sum_{\ell=0}^{i_2} (-1)^{i_2} \binom{i_2}{\ell} h_n^{0,\mathbf{b}} (i_1+\ell) \\ &-\sum_{\ell=0}^{i_2} (-1)^{i_2} \binom{i_2}{\ell} h_n^{0,\mathbf{b}} (i_1+\ell+1) =0.
	\end{split}
\end{equation} 
Finally, if $x=(x_1,2n+1)$ then~\eqref{Kinvboundaryodd} becomes
\begin{equation} \label{Kinvboundaryproof0topright}
-K^{-1}_n (x+(0,-1), \mathbf{b})  -K^{-1}_n(x+(1,0),\mathbf{b}). 
\end{equation}
 This means that we have
$$
x+(0,-1)=\big(x_1,x_1+2\frac{2n-x_1}{2} \big)
$$
and
$$
x+(1,0)=\big(x_1+1,x_1+1+2\frac{2n-x_1}{2} \big).
$$
Using~\eqref{Knb0}, we have that~\eqref{Kinvboundaryproof0topright} becomes
\begin{equation}
\begin{split}
&-\sum_{\ell=0}^{n-\frac{x_1}{2}} (-1)^{n-\frac{x_1}{2}} \binom{n-\frac{x_1}{2}}{\ell} h_n^{0,\mathbf{b}} (x_1+\ell)
-\sum_{\ell=0}^{n-\frac{x_1}{2}} (-1)^{n-\frac{x_1}{2}} \binom{n-\frac{x_1}{2}}{\ell} h_n^{0,\mathbf{b}} (x_1+1+\ell)\\
&=\sum_{\ell=0}^{n+1-\frac{x_1}{2}} (-1)^{n+1-\frac{x_1}{2}} \binom{n+1-\frac{x_1}{2}}{\ell} h_n^{0,\mathbf{b}} (x_1+\ell)
\end{split}
\end{equation}
where we have used~\eqref{Kinvboundaryoddrelation}.  We need to show the right side of the above equation equals zero.
Expanding the above term out using the formula for~\eqref{hnb0} gives 
\begin{equation}
\begin{split}
&\sum_{k=0}^n  \frac{p(n,k,0)}{2\pi \mathrm{i}} 
\int_{\Gamma_0} \text{d}r \frac{(1+r)^{n-k}(1+1/r)^{n+1-x_1/2}r^{2k}}{(1-r)r^i } 
- \sum_{k=0}^{n+1-x_1/2} \binom{n+1-\frac{x_1}{2}}{k} [k+1]_2 \\
&=\sum_{k=0}^n \frac{p(n,k,0)}{2\pi \mathrm{i}}  
\int_{\Gamma_0} \text{d}r \frac{(1+r)^{2n-k+1-x_1/2}}{(1-r) r^{n+1+x_1/2-2k }} 
-2^{n-x_1/2}\\ 
&=\sum_{k=0}^n p(n,k,0)
\sum_{j=0}^{n+\frac{x_1}{2} -2k} \binom{ 2n+1-\frac{x_1}{2}-k}{j}-2^{n-x_1/2},
\end{split}
\end{equation}
where we have used \eqref{integral2} for the second equality. We use \cref{thm:sumg} on the first term of the last line of the above equation and so the above equation is equal to zero as required.

We have thus verified \eqref{Kinvboundary} for all $x$, completing the proof.
\end{proof}

\section{Both vertices at the top boundary}
\label{sec:bothtop}
We will first establish a formula for $K_n^{-1}$ when both entries are top at the boundary using a series of lemmas in \cref{sec:bothtop}. 

\begin{rem}
\label{rem:pendant}
We will repeatedly use the following property of perfect matchings.
If a leaf $\ell$ is adjacent to a vertex $v$ via an edge with weight 1, then removing both $\ell$ and $v$ from the graph does not change the partition function.
\end{rem}

We will make use of the following result below, which follows immediately from \cref{thm:localstats} by multiplying both sides of the equation in \cref{thm:localstats} by $Z_G$.

\begin{prop}[Graphical Condensation~\cite{Kuo06}] \label{prop:Kuo}
	Let $G$ be a plane graph with four vertices $a,b,c,d$ that appear in that cyclic order on a face of $G$. Then 
	\begin{equation}
		Z_G Z_{G\backslash \{a,b,c,d \}} +Z_{G\backslash \{a,c \}} Z_{G\backslash \{b,d \}}=Z_{G\backslash \{a,b \}} Z_{G\backslash \{c,d \}}+Z_{G\backslash \{a,d \}} Z_{G\backslash \{b,c\}}.
	\end{equation}
\end{prop}

Recall that $Z_n^{\{v_1,\dots,v_m\}}$ is the partition function of the dimer model on $G_n$ with the vertices $\{v_1,\dots,v_m \}$ removed.
Introduce for $0 \leq i \leq  n-1$,
\begin{equation}
	T_n(i) = |K^{-1}_n((2i,2n+1),\mathbf{b})| 
	= \frac{Z_n^{\{(2i,2n+1),\mathbf{b}\}}}{Z_n},
\end{equation}
and for $0 \leq i < j \leq n-1$,
\begin{equation}
\label{defR}
	R_n(i,j) = |K^{-1}_n((2i,2n+1),(2j,2n+1))|
	= \frac{Z_n^{\{(2i,2n+1),(2j,2n+1)\}}}{Z_n},
\end{equation}
with $R_n(i,j)=-R_n(j,i)$ if $1\leq j<i\leq n-1$ and $R_n(i,i)=0$ for $0 \leq i \leq n-1$. We use the convention that $T_n(k)=0$ for $k\geq n$, $R_0=1$ and that $R_n(k,l)=0$ if $k \geq n$ or $l \geq n$.

\begin{lem}\label{lem:condensation}
For $1 \leq i,j\leq n-1$, we have
	\begin{equation}
		R_n(i,j)=R_{n-1}(i-1,j-1)+T_n(i)T_{n-1}(j-1)-T_{n-1}(i-1)T_n(j).
	\end{equation}
\end{lem}
\begin{proof}
	For the purpose of the proof, write $\mathbf{i}=(2i,2n+1)$ for $1 \leq i \leq n-1$ and let $\mathbf{a}=(0,2n+1)$. 
	 For $i<j$, we apply graphical condensation in~\cref{prop:Kuo}\footnote{Here the face in~\cref{prop:Kuo} is the boundary face.} which gives
	\begin{equation}\label{eq:Kuo}
		Z_n^{\{\mathbf{a},\mathbf{i},\mathbf{j},\mathbf{b}\}}Z_n = Z_n^{\{\mathbf{a},\mathbf{i}\}}Z_n^{\{\mathbf{j},\mathbf{b}\}}-Z_n^{\{\mathbf{a},\mathbf{j}\}}Z_n^{\{\mathbf{i},\mathbf{b}\}}+Z_n^{\{\mathbf{a},\mathbf{b}\}}Z_n^{\{\mathbf{i},\mathbf{j}\}}.
	\end{equation}
	 since $\mathbf{a},\mathbf{i},\mathbf{j},\mathbf{b}$ are in cyclic order on the face of TSSCPP.  Notice that removing $\mathbf{a}$ from $\mathtt{V}_n$ transforms the graph to $\mathtt{V}_{n-1}\backslash \{\mathbf{b} \}$ 
since the edges $((0,0),(1,1))$, $((0,1+2k),(0,2+2k))$, $((1,2+2k),(1,3+2k))$ for $0 \leq k \leq n-1$ must be matched by \cref{rem:pendant}; see \cref{fig:Kuo}. 
	
\begin{figure}[htbp!]
\begin{center}
	\includegraphics[height=5cm]{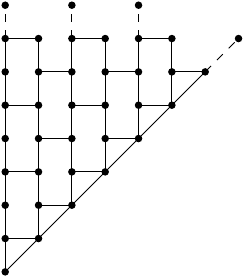}
	\hspace{5mm}
	\includegraphics[height=5cm]{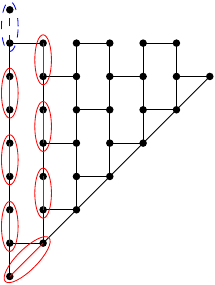}
\caption{The left figure shows adding pendant vertices to $\mathbf{a}=(0,7)$, $\mathbf{1}=(2,7)$, $\mathbf{2}=(4,7)$ and $\mathbf{b}$ from  $\mathtt{V}_3$ (here $n=3$) for considering $Z_3^{\{\mathbf{a},\mathbf{1},\mathbf{2},\mathbf{b}\}}$. These pendant edges are drawn in dashed.  Note that removing a vertex is the equivalent of adding a pendant edge to that vertex.
The right figure shows that adding a pendant vertex to $\mathbf{a}$ forces the blue dashed dimer as well as the red dimers. }
\label{fig:Kuo}
\end{center}
\end{figure}

	Hence, we have $Z_n^{\{\mathbf{a},\mathbf{i},\mathbf{j},\mathbf{b}\}}=Z_{n-1}^{\{\mathbf{i -1},\mathbf{j -1 }\}}$, $Z_n^{\{\mathbf{a},\mathbf{i}\}}=Z_{n-1}^{\{\mathbf{i-1},\mathbf{b}\}}$ and $Z_n^{\{\mathbf{a},\mathbf{b}\}}=Z_{n-1}$.  This means that the above equation is reduced to
\begin{equation}
	Z_{n-1}^{\{\mathbf{i-1},\mathbf{j-1}\}}Z_n = Z_{n-1}^{\{\mathbf{i-1},\mathbf{b}\}}Z_n^{\{\mathbf{j},\mathbf{b}\}}-Z_{n-1}^{\{\mathbf{j-1},\mathbf{b}\}}Z_n^{\{\mathbf{i},\mathbf{b}\}}+Z_{n-1}Z_n^{\{\mathbf{i},\mathbf{j}\}}.
	\end{equation}
	We divide the above equation by $Z_{n-1} Z_n$ and use the definitions of $R_n$ and $T_n$ to obtain the equation. A similar computation holds for $i>j$.  
\end{proof}

\begin{lem}\label{lem:Tn}
	For $0\leq i \leq n-1$, we have	$T_n(i)=(-1)^i p(n,i,0)$
	and $\sgn \, K^{-1}_n((2i,2n+1),\mathbf{b})=(-1)^n$. 
\end{lem}

\begin{proof}	
	Notice that $(2i,2n+1)=(2i,2(n-i)+2i+1)$.	
	The lemma follows from using the formula for $K^{-1}_n((2i,2n+1),\mathbf{b})$ given in~\eqref{Knb1} and computing explicitly as follows:
	\begin{equation}
		\begin{split}
			K^{-1}_n((2i,2n+1),\mathbf{b})&= \sum_{\ell=0}^{2i} (-1)^{n-i+l} \binom{n-i+\ell-1}{\ell}h_n^{1,\mathbf{b}}(2i-l) \\ 
			&=\sum_{k=0}^n  \frac{p(n,k,0)}{2\pi \mathrm{i}} \int_{\Gamma_0} \text{d}r \sum_{\ell=0}^{2i}(-1)^{n-i+\ell} \binom{n-i+\ell-1}{\ell} \frac{(1+r)^{n-k}}{r^{2i-2k-\ell+1}} \\
			&=(-1)^{n-i}\sum_{k=0}^n p(n,k,0) \frac{1}{2\pi \mathrm{i}} \int_{\Gamma_0} \text{d}r  \frac{(1+r)^{n-k+i-n}}{r^{2i-2k+1}} \\
			&=(-1)^{n-i} p(n,i,0),
		\end{split}
	\end{equation}
where the third equality uses the fact that the additional term proportional to $\pFq{2}{1}{1,1+i+n}{2+2i}{-r}$ does not have a singularity in $r$
and the fourth equality follows from \eqref{integral1}.  Since $\sgn \, p(n,i,0)=(-1)^i$, the result follows.  
\end{proof}

\begin{lem} \label{lem:signs}
For $0\leq i<j\leq n-1$, $\sgn \, K_n^{-1}((2i,2n+1),(2j,2n+1))=1$.   
\end{lem}

\begin{proof}
	Introduce an auxiliary edge directed from $(2j,2n+1)$ to $(2i,2n+1)$ for $i<j$. This introduces a new face, but by this choice, the number of counterclockwise edges on the graph remains odd and so is a valid Kasteleyn orientation.  Let $\tilde{K}_n$ denote the modified Kasteleyn matrix with this additional edge. Then, by \cref{thm:localstats} and the antisymmetry of $K_n$
	\begin{equation}\begin{split}
	0\leq	&\mathbb{P}[\mbox{Edge $((2i,2n+1),(2j,2n+1))$ in the modified graph}]\\&= \tilde{K}_n((2i,2n+1),(2j,2n+1)) \tilde{K}_n^{-1}((2j,2n+1),(2i,2n+1))\\&=-\tilde{K}_n^{-1}((2j,2n+1),(2i,2n+1))=\tilde{K}_n^{-1}((2i,2n+1),(2j,2n+1)).
	\end{split}
	\end{equation}
Since we have $\tilde{K}_n^{-1}((2i,2n+1),(2j,2n+1))={K}_n^{-1}((2i,2n+1),(2j,2n+1))$, the result follows. 
	\end{proof}

To get an explicit expression for $R_n(i,j)$ we use generating functions and require some further notation. Introduce 
\begin{equation}
		\mathbf{R}(u,v,w)= \sum_{n=0}^\infty \sum_{i=0}^n \sum_{j=0}^n R_n(i,j) u^i v^j w^n,
	\end{equation}
	and
	\begin{equation}
		\mathbf{T}_n(u) = \sum_{i=0}^{\infty} T_n(i)u^i.
	\end{equation}

\begin{lem}\label{lem:Req}
For $0 \leq i,j\leq n-1$ and $n\geq 1$,  
\begin{equation}
\begin{split}
R_n(i,j) &= \frac{1}{(2\pi \mathrm{i})^3} 
\int_{\Gamma_0} \frac{\text{d}u}{u}
\int_{\Gamma_0} \frac{\text{d}v}{v}
\int_{\Gamma_0}\frac{\text{d}w}{w}
\frac{1}{1-uwv} \frac{1}{u^iv^jw^n}	\\
		&\times\Bigg( \sum_{n=1}^\infty \mathbf{T}_n(u) \mathbf{T}_{n-1}(v) v w^n - \mathbf{T}_{n-1}(u) \mathbf{T}_{n}(v) u w^n \Bigg).
	\end{split}	
\end{equation}
\end{lem}

\begin{proof}
Multiply the recurrence in \cref{lem:condensation} by $u^iv^jw^n$ and taking sums gives
\begin{equation}
\begin{split}\label{lemproof:gf}
&\sum_{n=1}^\infty \sum_{i=1}^{n-1} \sum_{j=1}^{n-1} (R_n(i,j) u^i v^j w^n-  R_{n-1}(i-1,j-1)) u^i v^j w^n \\
=& \sum_{n=1}^\infty \sum_{i=1}^{n-1} \sum_{j=1}^{n-1} (T_n(i)T_{n-1}(j-1) -T_{n-1}(i-1)T_{n}(j))u^i v^j w^n.
\end{split}		
\end{equation}
Using $R_n(0,i)=T_{n-1}(i-1)$ and $R_n(i,j)=-R_n(j,i)$ we have that 
\begin{equation}
-\sum_{n=1}^\infty \sum_{i=1}^n R_n(i,0)u^i w^n=\sum_{n=1}^\infty u \mathbf{T}_{n-1}(u)w^n
\end{equation}
and 
\begin{equation}
\sum_{n=1}^\infty \sum_{j=1}^n R_n(0,j)u^j w^n=\sum_{n=1}^\infty v \mathbf{T}_{n-1}(v)w^n.
\end{equation}
Inserting the above two equations into~\eqref{lemproof:gf} gives
\begin{equation}
\mathbf{R}(u,v,w)(1-u w v) = 1+ \sum_{n=1}^\infty  \mathbf{T}_{n}(u) \mathbf{T}_{n-1}(v)v w^n-\sum_{n=1}^\infty u \mathbf{T}_{n-1}(u) \mathbf{T}_{n}(u) w^n.
\end{equation}
We divide both sides by $(1-u w v)$ and it remains to extract out the coefficient of $u^i v^j w^n$ in the above equation to prove the lemma. 
\end{proof}

\begin{lem} \label{lem:intbfT}
For $|u|<1$ 			
\begin{equation}
\begin{split}
\mathbf{T}_{n}(u)= &{\frac{1}{(2\pi \mathrm{i})^3}} 
\int_0^\infty \text{d}q 
\int_0^\infty \text{d}p 
\int_{\Gamma_0}{\text{d}t} 
\int_{\Gamma_0}{\text{d}z}
\int_{S}  \text{d}s \; 
e^{-p-q+\frac{p^2sq}{t^2}+t} \\
& \times\frac{ p^{2+n}q^{3+2n}(1+z)^{2n+2}(2+z)}{t^{3(1+n)}z^{1+n}(pq^2s(1+z)^2-t^3 uz)},
\end{split}
\end{equation}
	where the $s$-contour ${S}$ is chosen so that $\frac{ |t^3zu|}{p q^2|s(1+z)^2|}<1$.
\end{lem}

\begin{proof}
Notice that $\mathbf{T}_n(u)=\sum_{k = 0}^{{n}} T_n(k)u^k$ 
since $T_n(k)=0$ for $k \geq n$ and so we use integral expressions to find an explicit expression using geometric sums. 
We now use the expressions for binomial coefficients and factorials given in~\eqref{integral2},~\eqref{integral3}, and~\eqref{integral4}, as explained below, in the formula for $T_n(i)$ given in  \cref{lem:Tn} by applying these integrals to terms in the expression for $(-1)^kp(n,k,0)$~\eqref{eq:pnkl}. We obtain
			\begin{equation}\label{eq:firstTnintegral}
				\begin{split}
T_n(k)u^k&=\frac{1}{(2\pi \mathrm{i})^3} 
\int_0^\infty \text{d}q 
\int_0^\infty \text{d}p
\int_{\Gamma_0}  \frac{\text{d}z}{z}
\int_{\Gamma_0} \frac{\text{d}t}{t}
\int_{\Gamma_0} \frac{\text{d}s}{s} \;
e^{-p-q+s+t} \\
					&\times \frac{p^{n+k+1}q^{2n-k+1}(1+z)^{2n-2k}(2+z)}{s^k t^{3n-k+2}z^{n-k}}u^k.
				\end{split}			\end{equation}
				In the above equation, the integrals with respect to $p$ and $q$ are from the two factorial terms in the numerator of $(-1)^kp(n,k,0)$ using~\eqref{integral3}, while the integrals respect to $s$ and $t$ are from the two factorials in the denominator of $(-1)^kp(n,k,0)$ using~\eqref{integral4}. For the integral with respect to $z$, this comes from rewriting the term $(3n-3k+2)C_{n-k}$ in $(-1)^k p(n,k,0)$ as
				\begin{equation}
					(3n-3k+2)C_{n-k} = 2 \binom{2n-2k}{n-k}+\binom{2n-2k}{n-k+1},
				\end{equation}
				and expressing each of these binomial terms as integrals using~\eqref{integral1}.

We choose the $s$-contour in~\eqref{eq:firstTnintegral} so that 
				\begin{equation}
					\frac{p |tzu|}{q|s(1+z)^2|}<1,
				\end{equation}
and evaluate the geometric sum 
\begin{equation}
	\begin{split}
\sum_{k=1}^\infty		T_n(k)u^k&=\frac{1}{(2\pi \mathrm{i})^3} \int_0^\infty  \text{d}q
\int_0^\infty \text{d}p
\int_{\Gamma_0}  {\text{d}z}
\int_{\Gamma_0}   {\text{d}t}
\int_{S}{\text{d}s} \\
&		\frac{ e^{-p-q+s+t} p^{1+n}q^{2+2n}(1+z)^{2n+2}(2+z)}{t^{3(1+n)}z^{1+n}(qs(1+z)^2-pt uz)}.
				\end{split}			
\end{equation}
Making the change of variables $s\mapsto s p^2 q/t^2$ gives the result. 
\end{proof}

In what follows below, we use the notation $\text{d} \mathbf{p}=\text{d}p_1 \text{d}p_2$ and similarly for the variables $\mathbf{q},\mathbf{t},\mathbf{z},\mathbf{s}$
for compactness of notation. 

\begin{lem}\label{lem:finalR}
	For $0\leq i,j \leq n-1$,
\begin{equation}
\begin{split}
&	R_n(i,j)=\frac{1}{(2\pi \mathrm{i})^6}
\iint_{\mathbb{R}_+^2}			\text{d}\mathbf{q}
\iint_{\mathbb{R}_+^2}			\text{d}\mathbf{p}
\iint_{\Gamma_0^2}			\text{d}\mathbf{z}
\iint_{\Gamma_0^2}			\text{d}\mathbf{t}
\iint_{\Gamma_0^2}			\text{d}\mathbf{s}
\\
			&\times e^{-p_1-p_2-q_1-q_2+\frac{p_1^2s_1q_1}{t_1^2}+ \frac{p_2^2s_2q_2}{t_2^2}+t_1+t_2}\frac{ p_1^{1-i+n}p_2^{1-j+n}q_1^{1-2i+2n}q_2^{1-2j+2n}}{s_1^{1+i}s_2^{1+j} t_1^{3-3i+3n}t_2^{3-3j+3n} } \frac{s_1-s_2}{s_1s_2-1} \\&\times \frac{(1+z_1)^{2n-2i}(1+z_2)^{2n-2j}(2+z_1)(2+z_2)}{z_1^{n-i+1}z_2^{n-j+1}}.
		\end{split}	
		\end{equation}
\end{lem} 
	
\begin{proof}
This lemma follows from a computation starting with the formula for $R_n(i,j)$ given in \cref{lem:Req}.  For brevity, we only present the computation of the first term $\mathbf{T}_{n}(u)\mathbf{T}_{n-1}(v) v w^n$; the computation for the other term is completely analagous.  

Introduce 
		\begin{equation}
			\tilde{R}= \frac{p_1p_2 q_1^2 q_2^2 (1+z_1)^2(1+z_2)^2}{t_1^3 t_2^3 z_1 z_2}.
		\end{equation}
		We can use the integral expression for $\mathbf{T}_n(u)$ given in \cref{lem:intbfT} to find that 
		\begin{equation}
			\begin{split}
&	\sum_{n=1}^\infty 	\mathbf{T}_n(u)\mathbf{T}_{n-1}(v) v w^n =  \frac{1}{(2\pi \mathrm{i})^6}
\iint_{\mathbb{R}_+^2}			\text{d}\mathbf{q}
\iint_{\mathbb{R}_+^2}			\text{d}\mathbf{p}
\iint_{\Gamma_0^2}			\text{d}\mathbf{z}
\iint_{\Gamma_0^2}			\text{d}\mathbf{t}
\iint_{S_1 \times S_2} 			\text{d}\mathbf{s}
\\
				&\times v\sum_{n=1}^\infty w^n \tilde R^n
				\frac{e^{-p_1-q_1+\frac{p_1^2s_1q_1}{t_1^2}+t_1-p_2-q_2+\frac{p_2^2s_2q_2}{t_2^2}+t_2}p_1^{2}p_2q_1^{3}q_2(1+z_1)^2 (2+z_1)(2+z_2)}{t_1^3 z_1(p_1q_1^2s_1(1+z_1)^2-t_1^3 uz_1)(p_2q_2^2s_2(1+z_2)^2-t_2^3 vz_2)}
			\end{split}		\end{equation}
where the $s_1$-contour ${S_1}$ is chosen so that $\frac{ |t_1^3z_1u|}{p_1 q_1^2|s_1(1+z_1)^2|}<1$ and  the $s_2$-contour ${S_2}$ is chosen so that $\frac{ |t_2^3z_2u|}{p_2 q_2^2|s_2(1+z_2)^2|}<1$. 
			We substitute the above expression into the right side of the equation in \cref{lem:Req} which gives
\begin{equation} \label{eq:prebigcontour}
\begin{split}
&\frac{1}{(2\pi \mathrm{i})^3} 
\int_{\Gamma_0} \frac{\text{d}u}{u}
\int_{\Gamma_0} \frac{\text{d}v}{v}
\int_{\Gamma_0}\frac{\text{d}w}{w}
\frac{1}{1-uwv} \frac{1}{u^iv^jw^n}	\sum_{n=1}^\infty \mathbf{T}_n(u) \mathbf{T}_{n-1}(v) v w^n\\
&=\frac{1}{(2\pi \mathrm{i})^3} 
\int_{\Gamma_0} \frac{\text{d}u}{u}
\int_{\Gamma_0} \frac{\text{d}v}{v}
\int_{\Gamma_0} \frac{\text{d}w}{w}
\frac{1}{1-uwv} \frac{1}{u^iv^jw^n}\frac{\tilde Rw v}{1-\tilde Rw}\\
					&\times
\frac{1}{(2\pi \mathrm{i})^6}
\iint_{\mathbb{R}_+^2}			\text{d}\mathbf{q}
\iint_{\mathbb{R}_+^2}			\text{d}\mathbf{p}
\iint_{\Gamma_0^2}			\text{d}\mathbf{z}
\iint_{\Gamma_0^2}			\text{d}\mathbf{t}
\iint_{{S_1\times S_2}} \text{d}\mathbf{s}
\\
				&	\times			\frac{e^{-p_1-q_1+\frac{p_1^2s_1q_1}{t_1^2}+t_1-p_2-q_2+\frac{p_2^2s_2q_2}{t_2^2}+t_2}p_1^{2}p_2q_1^{3}q_2(1+z_1)^2 (2+z_1)(2+z_2)}{t_1^3 z_1(p_1q_1^2s_1(1+z_1)^2-t_1^3 uz_1)(p_2q_2^2s_2(1+z_2)^2-t_2^3 vz_2)}.
				\end{split}
			\end{equation}
			In the above expression, we rearrange the integrals and compute the $w$-integral first (no extra contribution is picked up). We push the contour through $\infty$ picking up residue contributions at $w=1/R$ and $w=1/(u v)$.  The latter will not contribute since when computing the residue at $w=1/(u v)$, there will be no residue at $u=0$.  This means that the above equation becomes
\begin{equation}
\begin{split}
&	\frac{1}{(2\pi \mathrm{i})^8}  
\int_{\Gamma_0} \frac{\text{d}u}{u}
\int_{\Gamma_0}\frac{\text{d}v}{v}
\iint_{\mathbb{R}_+^2}			\text{d}\mathbf{q}
\iint_{\mathbb{R}_+^2}			\text{d}\mathbf{p}
\iint_{\Gamma_0^2}			\text{d}\mathbf{z}
\iint_{\Gamma_0^2}			\text{d}\mathbf{t}
\iint_{S_1 \times S_2}	\text{d}\mathbf{s}
\\
				&\times		\frac{1}{\tilde R-u v} \frac{\tilde R^{n+1}v}{u^iv^j}		\frac{e^{-p_1-q_1+\frac{p_1^2s_1q_1}{t_1^2}+t_1-p_2-q_2+\frac{p_2^2s_2q_2}{t_2^2}+t_2}p_1^{2}p_2q_1^{3}q_2(1+z_1)^2 (2+z_1)(2+z_2)}{t_1^3 z_1(p_1q_1^2s_1(1+z_1)^2-t_1^3 uz_1)(p_2q_2^2s_2(1+z_2)^2-t_2^3 vz_2)}.
			\end{split}\end{equation}
				Next, we rearrange integrals to compute the $u$-integral  by computing the residue contributions at $u=R/v$ and $u=p_1 q_1^2 s_1(1+z_1)^2/(t_1^3 z_1)$ and so we obtain
\begin{equation}
\begin{split}\label{eq:bigcontourexpression}
&	\frac{1}{(2\pi \mathrm{i})^7}  
\int_{\Gamma_0}\frac{\text{d}v}{v}
\iint_{\mathbb{R}_+^2}			\text{d}\mathbf{q}
\iint_{\mathbb{R}_+^2}			\text{d}\mathbf{p}
\iint_{\Gamma_0^2}			\text{d}\mathbf{z}
\iint_{\Gamma_0^2}			\text{d}\mathbf{t}
\iint_{S'_1 \times S_2}			\text{d}\mathbf{s}
\\
					&\times		 \frac{\tilde R^{n-i}}{v^{j-1-i}}		\frac{e^{-p_1-q_1+\frac{p_1^2s_1q_1}{t_1^2}+t_1-p_2-q_2+\frac{p_2^2s_2q_2}{t_2^2}+t_2}p_1^{2}p_2q_1^{3}q_2(1+z_1)^2 (2+z_1)(2+z_2)}{t_1^3 z_1(p_1q_1^2s_1(1+z_1)^2-t_1^3 vz_1)(p_2q_2^2s_2(1+z_2)^2-t_2^3 vz_2)}\\
&+\frac{1}{(2\pi \mathrm{i})^7}  
\int_{\Gamma_0}\frac{\text{d}v}{v}
\iint_{\mathbb{R}_+^2}			\text{d}\mathbf{q}
\iint_{\mathbb{R}_+^2}			\text{d}\mathbf{p}
\iint_{\Gamma_0^2}			\text{d}\mathbf{z}
\iint_{\Gamma_0^2}			\text{d}\mathbf{t}
\iint_{S_1 \times S_2}			\text{d}\mathbf{s}
\\
					&\times		 \frac{t_1^{3i}z_1^{i-1}e^{-p_1-q_1+\frac{p_1^2s_1q_1}{t_1^2}+t_1-p_2-q_2+\frac{p_2^2s_2q_2}{t_2^2}+t_2}p_1^{2}p_2q_1^{3}q_2(1+z_1)^2 (2+z_1)(2+z_2)}{(p_1 q_1^2s_1(1+z_1)^2)^{i}(p_1q_1^2s_1(1+z_1)^2v-R t_1^3 z_1)v^{j+1} t_1^3 z_1(p_2q_2^2s_2(1+z_2)^2-t_2^3 vz_2)},
			\end{split}\end{equation}
	where $S_1'$ is chosen so that $\frac{ |t_1^3z_1v|}{p_1 q_1^2|s_1(1+z_1)^2|}<1$.  We compute each of the above terms separately.  For the first term in~\eqref{eq:bigcontourexpression}, we rearrange the integrals to compute the $v$-integral and there is a residue at $v=0$, provided that $j>i+1$, and further residue contributions at $v=t_1^3 \tilde R z_1/(p_1q^2_1 s_1(1+z_1)^2)$ and $v=t_2^3  z_2/(p_2q^2_2 s_2(1+z_2)^2)$ due to exchanging contours.  Summing up these contribution gives zero because there is no residue at $s_1=0$ after simplification (we omit the details of this computation). Thus,~\eqref{eq:prebigcontour} is equal to the second term in~\eqref{eq:bigcontourexpression}. For the second term in ~\eqref{eq:bigcontourexpression}, we perform the same computational steps as mentioned for the first term and we arrive at
\begin{equation}
\begin{split}
&\frac{1}{(2\pi \mathrm{i})^6}  
\iint_{\mathbb{R}_+^2}			\text{d}\mathbf{q}
\iint_{\mathbb{R}_+^2}			\text{d}\mathbf{p}
\iint_{\Gamma_0^2}			\text{d}\mathbf{z}
\iint_{\Gamma_0^2}			\text{d}\mathbf{t}
\iint_{\Gamma_0^2}			\text{d}\mathbf{s}
 e^{-p_1-q_1+\frac{p_1^2s_1q_1}{t_1^2}+t_1-p_2-q_2+\frac{p_2^2s_2q_2}{t_2^2}+t_2}\\
		&\times		\frac{ p_1^{1-i+n} p_2^{1-j+n} q_1^{1-2i+2n}q_2^{1-2j+2n}}{s_1^{1+i}s_2^j t_1^{3-3i+3n}t_2^{3-3j+3n}(s_1s_2-1)} \frac{(1+z_1)^{2n-2i}}{z_1^{n-i+1}} \frac{(1+z_2)^{2n-2j}}{z_2^{n-j+1}}(2+z_1)(2+z_2).
	\end{split}
	\end{equation}
We have now evaluated~\eqref{eq:prebigcontour} and the second term on right side of the equation in \cref{lem:Req} follows by symmetry. 
	\end{proof}

\section{Proof of the main result} 
\label{sec:main}

We will now use the results from the previous sections to prove \cref{thm:main}.
For the proof, we will need the following result.

\begin{prop} 
\label{prop:lastclaim}
For $0\leq j \leq 2n-1$
\begin{equation} 
	-K^{-1}_n((0,0),(j,j+1))=K^{-1}_n((j,j+1),(0,0))=K^{-1}_n((j,j+1),\mathbf{b}).
\end{equation}
\end{prop}

\begin{proof}
The first equality is from antisymmetry of $K_n$.  For the second equality, notice that $K^{-1}_n((j,j+1),(0,0))$ represents the ratio of a signed count on the graph formed from removing the vertices $(j,j+1)$ and $(0,0)$ along with their incident edges from the TSSCPP of size $n$ and $Z_n$.  The sign  arises from having an even number of counterclockwise edges around the face which surrounds the removed vertex $(j,j+1)$. For each of these dimer configurations (which is part of the signed count) on $\mathtt{V}_n\backslash \{(0,0),(j,j+1)\}$, remove all dimers incident to the vertices $\{(i,i):0 \leq i \leq 2n-1\}$.  Then, there is a unique way to extend to a dimer configuration   on $\mathtt{V}_n\backslash \{ \mathbf{b},(j,j+1)\}$ as required. Note that this operation does not change the sign associated to each dimer configuration. 
\end{proof}

\begin{proof}[Proof of \cref{thm:main}] 
		We will first focus on the formula \eqref{tn11} for $t_n^{1,1}(i,j)$. The formulas for $t_n^{k,l}(i,j)$ for $(k,l) \not = (1,1)$ are obtained by very similar computations. 
		Once these are found, the rest of entries are obtained from the matrix equations $K_n.K_n^{-1}=K_n^{-1}.K_n =\mathbbm{I}$. 

		To obtain $t_n^{1,1}(i,j)$, first observe that $t_n^{1,1}$ is a linear combination of binomial coefficients for $R_n(k,l)$ by using the formula $K_n.K_n^{-1}=K_n^{-1}.K_n =\mathbbm{I}$.  Indeed, notice that from the equation $K_n.K_n^{-1}=\mathbbm{I}$ entrywise we have 
		\begin{equation}
			K_n^{-1}((i_1,i_1+2i_2+1), y )=-K_n^{-1}((i_1-1,i_1+2i_2+2), y )-K_n^{-1}((i_1,i_1+2i_2+3), y)
		\end{equation}
		for $1 \leq i_1 \leq 2n$, $0\leq i_2 \leq n-\lceil i_1/2\rceil -1$ and $y=(j_1,j_1+2j_2+1)$ (a similar relation holds for $y$).  Using the above equation, we see that 
		\begin{equation}
			K^{-1}_n((i,i+1),(j,j+1))=	t_n^{1,1}(i,j) = \sum_{k_1=0}^{\lfloor \frac{i}{2} \rfloor} \sum_{k_2=0}^{\lfloor \frac{j}{2} \rfloor} (-1)^{k_1+k_2} \binom{n-k_1}{i-2k_1} \binom{n-k_2}{i-2k_2} R_n(k_1,k_2),
		\end{equation}
		since $R_n(k_1,k_2)=K^{-1}_n((2k_1,2n+1),(2k_2,2n+1))$ by \eqref{defR} and \cref{lem:signs}.   It suffices to evaluate the right side of the above equation and show that it equals $t_n^{1,1}(i,j)$ as given in~\eqref{tn11}.
		To do so, notice that by using the residue theorem
		\begin{equation}
			\frac{1}{2\pi \mathrm{i}} \int_{\Gamma_0} \frac{\text{d}t}{t}\frac{1}{t^{3n+2-2k}} e^{\frac{p^2 s q}{t^2}+t} = \sum_{\ell=0}^\infty \frac{(p^2 s q)^\ell }{\ell! (3n+2\ell+2-3k)!}.
		\end{equation}
		Using the above two equations, the formula for $R_n$ found in \cref{lem:finalR}, and writing binomial coefficients above as contour integrals in the $\mathbf{r}$ variables using \eqref{integral1}, we obtain
		\begin{equation}
			\begin{split}
				& \sum_{\ell_1=0}^\infty\sum_{\ell_2=0}^\infty \sum_{k_1=0}^{\lfloor \frac{i}{2} \rfloor} \sum_{k_2=0}^{\lfloor \frac{j}{2} \rfloor} \frac{1}{\ell_1! \ell_2!} \frac{(-1)^{k_1+k_2}}{(2\ell_1+3n+2-3k_1)! (2\ell_1+3n+2-3k_2)! } \\
&\times\frac{1}{(2 \pi \mathrm{i})^6} 
\iint_{\mathbb{R}_+^2}			\text{d}\mathbf{q}
\iint_{\mathbb{R}_+^2} 	\text{d}\mathbf{p}
\iint_{\Gamma_0^2}		\text{d}\mathbf{z}
\iint_{\Gamma_0^2}		\text{d}\mathbf{s}
\iint_{\Gamma_0^2}		\text{d}\mathbf{r}
				\frac{(1+r_1)^{n-k_1}}{r_1^{i-2k_1+1}} \\
&\times \frac{(1+r_2)^{n-k_2}}{r_2^{j-2k_2+1}} p_1^{1+2\ell_1-k_1+n} p_2^{1+2\ell_2-k_2+n} \frac{ q_1^{1-2k_1+\ell_1+2n}q_2^{1-2k_2+\ell_2+2n}}{s_1^{1+k_1-\ell_1} s_2^{1+k_2-\ell_2}}  \\
&\times \frac{s_1-s_2}{s_1 s_2-1}  \frac{(1+z_1)^{2n-2k_1}}{z_1^{n-k_1+1}}\frac{(1+z_2)^{2n-2k_2}}{z_2^{n-k_2+1}}(2+z_1)(2+z_2)e^{-p_1-p_2-q_1-q_2} .
			\end{split}
		\end{equation}
		In the above formula, we use the standard expressions for integrals to give closed forms. That is, we use~\eqref{integral3} to convert the integrals with respect to the variables $\mathbf{p}$ and $\mathbf{q}$ and use ~\eqref{integral4} to convert the integral with respect to the variable $\mathbf{s}$.  The integral with respect to $\mathbf{z}$ is converted using~\eqref{integral1}; see the discussion after~\eqref{eq:firstTnintegral} for the computation in reverse.
				Taking the change of summation $k_\varepsilon\mapsto \ell_\varepsilon-k_\varepsilon$ for $\varepsilon \in\{1,2\}$ gives the formula for $t_n^{1,1}(i,j)$ given in~\eqref{tn11}.  	

		We now explain how to get the remaining formulas from~\eqref{tn11}. From expanding the equation $K_n . K_n^{-1}=\mathbbm{I}$ entrywise, we have the formula for $i \geq 2$ 
		\begin{equation}
			\begin{split}
				&-K_n^{-1}((i,i),(j,j+1))+K_n^{-1}((i-2,i-2),(j,j+1))+K_n^{-1}((i-1,i),(j,j+1)) \\
				&K_n^{-1}((i-2,i-1),(j,j+1))= \mathbbm{I}_{(i,i)=(j,j+1)}=0.
			\end{split}		\end{equation}
We now rearrange the above equation to obtain
			\begin{equation}
				K_n^{-1}((i,i),(j,j+1))=\sum_{r=0}^{i-1} K_n^{-1}((r,r+1),(j,j+1))+[i+1]_2 K^{-1}_n((0,0),(j,j+1)).
			\end{equation}
			Since $K_n^{-1}((r,r+1),(j,j+1))=t_n^{(1,1)}(r,j)$, we can push the sum through to the terms in the integral in~\eqref{tn11} and evaluate. This computation and \cref{prop:lastclaim} gives~\eqref{tn01}.
			The formulas in~\eqref{tn10} and~\eqref{tn00} follow from similar computations. 

			It is not hard to see that~\eqref{Kinv11}, ~\eqref{Kinv00} and~\eqref{Kinv10} follow from a straightforward applications of $K_n.K_n^{-1}=K_n^{-1}.K_n=\mathbbm{I}$.  
	\end{proof}

\section{Proofs of the combinatorial identities} 
\label{sec:identities}

We now give the proofs of the combinatorial results used in \cref{sec:vertexb}.
 
\begin{proof}[Proof of \cref{thm:sumf}]
A minor issue is that the natural upper limit for the sum is $k=n+1$, not $k=n$, although this
is what it seems at first glance. Moreover, $f_n(n+1) = -(-1)^n/2$. Therefore, we 
need to prove that $\sum_k f_n(k) = 1/2$.
\texttt{Maple} solves this using the Wilf-Zeilberger algorithm. Then we need to check that
\[
f_n(k) - f_{n+1}(k) = h(n,k+1) - h(n,k),
\]
where the certificate is
\begin{multline*}
h(n,k) = (-1)^k 2^{n-k+1} (n+2-k) (n-k+1) (5n-3k+8) \\
\times \frac{(2n+2-k)! (n+1+k)! (2n+2-2k)!}{(k-1)! (n+2-k)!^2 (3n+5-k)!}.
\end{multline*}
This guarantees that $\sum_k f_n(k)$ is constant. We now check 
that $f_0(0) + f_0(1) = 1 - 1/2 = 1/2$, completing the proof.
\end{proof}

Before proving \cref{thm:sumg}, we need another identity in the sequel, which we prove now.

\begin{thm}
\label{thm:sumf'}
Let $n \in \mathbb{N}$ and $0 \leq i \leq n-1$. Define
\[
f'_{n,i}(k) = (-1)^{k+1}  \frac{(3n-3k+2)(n+k+1)! (2n-k+1)! (2n-i-k)!}{k! (3n-k+2)! (n+i-2k+1)! (n+k-2i)!} C_{n-k}.
\]
Then $\sum_{k=0}^n f'_{n,i}(k) = 0$.
\end{thm}

\begin{proof}
Let $F'_{n,i}$ be the required sum. Then Zeilberger's algorithm implemented in the  
\texttt{Mathematica} package {\em fastZeil}~\cite{paule-schorn-1995} gives the recurrence
\begin{multline}
\label{iden-recur}
2(i+1)(2n+1-i)(2n+1-2i) F'_{n,i} + (3n+2-i)(n-1-i)(n+2+i) F'_{n,i+1} \\
= (-1)^n \frac{12 (n+1)^2 (n-i)!}{(2n-2i)! (i-n)!}.
\end{multline}
To prove \eqref{iden-recur}, we define the certificate,
\[
R(k,i)=\frac{6 k (i-n) (-2 k+2 n+1) (-k+2 n+2) (-2 i+k+n)}{(-3 k+3 n+2) (i-2 k+n+2)},
\]
and verify that
\begin{multline*}
2(i+1)(i-2 n-1)(2 i-2 n-1)f'_n(k,i) + 
(i-3 n-2)(i-n+1)(i+n+2)f'_n(k,i+1) \\
= \Delta_k \left( f'_n(k,i) R(k,i) \right),
\end{multline*}
where $\Delta_k$ is the forward difference operator defined by $\Delta_k(f(k)) = f(k+1) - f(k)$.

In \eqref{iden-recur}, notice that $1/(i-n)! = 0$ and all other terms on the right hand side are well-defined for $0 \leq i \leq n-1$. Therefore, the right hand side is zero for our region of interest. Because of the factor $(n-1-i)$ in the second term on the left hand side, setting $i=n-1$ also gives $F_{n,n-1} = 0$. (This can also be checked independently using Zeilberger's algorithm.) Then, the recurrence above guarantees that $F_{n,i} = 0$ for $0 \leq i \leq n-1$.
\end{proof}

\begin{proof}[Proof of \cref{thm:sumg}]
First, note that
\begin{multline}
\label{jsum}
\sum_{j=0}^{n+i-2k}  \binom{2n+1-i-k}{j} + \sum_{j=0}^{n+k-2i} \binom{2n+1-i-k}{j} \\
= \sum_{j=0}^{n+i-2k}  \binom{2n+1-i-k}{j} + \sum_{j=0}^{n+k-2i} \binom{2n+1-i-k}{2n+1-i-k-j} \\
= \sum_{j=0}^{n+i-2k}  \binom{2n+1-i-k}{j} + \sum_{j=n+1-2k+i}^{2n+1-i-k} \binom{2n+1-i-k}{j} \\ 
= \sum_{j=0}^{2n+1-i-k}  \binom{2n+1-i-k}{j} = 2^{2n+1-i-k}.
\end{multline}
Using \eqref{jsum}, add the left hand sides of \eqref{id1} and \eqref{id2} to get
\begin{multline*}
\sum_{k=0}^n \frac{1}{2^{n-k}} f_n(k) 
\left( \sum_{j=0}^{n+i-2k}  \binom{2n+1-i-k}{j} + \sum_{j=0}^{n+k-2i} \binom{2n+1-i-k}{j} \right) \\
= \sum_{k=0}^n \frac{1}{2^{n-k}} f_n(k)  2^{2n+1-i-k} = 2^{n-i} (1 + (-1)^n),
\end{multline*}
where we have used \cref{thm:sumf} in the last step.
So we have shown that the sum of the left hand sides is what we want.

We now prove \eqref{id1}. 
Our first task is to prove $G_{n,i} = 2G_{n,i+1}$.
Using Wegschaider's algorithm~\cite{wegschaider-1997} in the \texttt{Mathematica} package {\em MultiSum}, we get the recurrence 
\begin{equation}
\label{recur}
g_{n,i}(k,j) - 2 g_{n,i+1}(k,j) = \Delta_j \left[
g_{n,i+1}(k,j) - g_{n,i}(k,j) \right],
\end{equation}
where $\Delta_j$ is the forward difference operator defined above.
This is easily verified by a computation. Now, we perform the $j$-sum on both sides of \eqref{recur}. Since the right hand side telescopes, 
we obtain
\begin{align*}
&\sum_{j=0}^{n+i-2k} \left( g_{n,i}(k,j) - 2 g_{n,i+1}(k,j) \right) \\
&= g_{n,i+1}(k,n+i+1-2k) - g_{n,i}(k,n+i+1-2k) \\
&= (-1)^{k+1} \frac{(3n-3k+2)(n+k+1)! (2n-k+1)! (2n-i-k)!}{k! (3n-k+2)! (n+i-2k)! (n+k-2i)!} C_{n-k}.
\end{align*}
We perform the $k$-sum on both sides of the above equation and compensate for the $j=n+i+1-2k$ term to obtain that $G_{n,i} - 2G_{n,i+1}$ is equal to
\begin{align*}
& \sum_{k=0}^n (-1)^{k+1}  \frac{(3n-3k+2)(n+k+1)! (2n-k+1)! (2n-i-k)!}{k! (3n-k+2)! (n+i-2k)! (n+k-2i)!} C_{n-k}  \\
& - 2\sum_{k=0}^n (-1)^{k} \frac{(3n-3k+2)(n+k+1)! (2n-k+1)! (2n-i-k)!}{k! (3n-k+2)! (n+i-2k+1)! (n+k-2i-1)!} C_{n-k} \\
& = \sum_{k=0}^n f'_{n,i}(k),
\end{align*}
where $f'$ is defined above. Now, using \cref{thm:sumf'}, we obtain that this is equal to $0$.

All that remains to be done is to prove that $G_{n,n} = 1$.
But this is easily performed since the $j$-sum becomes
\[
\sum_{j=0}^{2n-2k} \binom{n-k+1}{j} = 2^{n-k+1},
\]
and we now use \cref{thm:sumf}.
This completes the proof.
\end{proof}

\section{Boundary recurrences} 
\label{sec:additionrec}

Recall that $g_n^{\mathbf{b}}$ is given in \eqref{eq:Kinvboundarycount}
and define, for $0 \leq i<j \leq 2n-1$,
\begin{equation}
	g_n(i,j)=|K^{-1}_n((i,i),(j,j))|=-K^{-1}_n((i,i),(j,j))
	=\frac{Z_n^{\{(i,i),(j,j)\}}}{Z_n},
\end{equation}
with $g_n(j,i)=-g_n(i,j)$ and $g_n(i,i)=0$.  The relevant signs given above are evaluated by using a similar argument given in \cref{lem:signs}; we omit this computation.  

We obtain some formulas immediately.

\begin{lem}\label{lem:immediateg}
	We have that for $2 \leq j \leq 2n-1$,
	\begin{equation}
		g_n^{\mathbf{b}}(1)=\frac{Z_{n-1}}{Z_n},
	\end{equation}
\begin{equation}
	g_n(1,j)=\frac{Z_{n-1}}{Z_n} g_{n-1}^{\mathbf{b}}(j-2),
	\end{equation}
	and $g_n^{\mathbf{b}}(0)=1$.
\end{lem}

\begin{proof}
	The first two equations are immediate from writing the definitions of $g_n^{\mathbf{b}}(1)$ and $g_n(1,j)$ as a ratio of partition functions, see also~\eqref{eq:Kinvboundarycount}, and removing the vertex $(1,1)$ from the graph forces edges to be covered by dimers; see \cref{fig:forced1}.

\begin{figure}[htbp!]
\begin{center}
	\includegraphics[height=5cm]{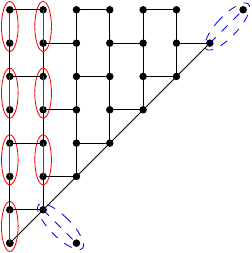}
	\caption{ The blue dashed dimers represents the removal of the vertices $(1,1)$ and $\mathbf{b}$ from the graph. The red dimers are those that are induced when the vertex $(1,1)$ is removed.  }
	\label{fig:forced1}
\end{center}
\end{figure}

	For the last equation, for each dimer configuration of $Z_n$ on $\mathtt{V}_n$, remove all dimers incident to the vertices $\{(i,i), 0 \leq i \leq 2n-1 \}$.  For each of these dimer configurations, there is a unique way to extend to a dimer configuration on $\mathtt{V}_n$ and $\mathtt{V}_n\backslash \{(0,0),\mathbf{b} \}$.  This shows that $Z_n=Z_n^{\{(0,0),\mathbf{b}\}}$ as required.  
\end{proof}

We then have the following system of equations for $g_n(i,j)$ and $g_n^\mathbf{b}(j)$:

\begin{thm}
\label{thm:boundaryrec2}
For $0 \leq i < j \leq 2n-1$, we have the recurrence
	\begin{equation}\label{eq:gnij}
	g_n(i,j)= \begin{cases}
		1-g_n^{\mathbf{b}} (j) & \mbox{if }i=0, j>0, \\[0.2cm]
		\ds \frac{Z_{n-1}}{Z_n}g_{n-1}^\mathbf{b} (j-2) & \mbox{if }i=1,j>1, \\[0.2cm]
		g_{n-1}(i-2,j-2) - g_n^{\mathbf{b}}(j) g_{n-1}^{\mathbf{b}} (i-2) + 	g_n^{\mathbf{b}}(i) g_{n-1}^{\mathbf{b}} (j-2) & \mbox{if }i,j \geq 2.
	\end{cases}
	\end{equation}

\end{thm}

\begin{thm}
\label{thm:boundaryrec1}
For $0 \leq j \leq 2n-1$, we have the recurrence
	\begin{equation}\label{eq:gbj}
g_n^\mathbf{b}(j)=\begin{cases}
			1 & \mbox{if }j=0, \\[0.2cm]
			\ds \frac{Z_{n-1}}{Z_n} & \mbox{if }j=1,   \\[0.2cm]
		\ds\frac{Z_{n-1}}{Z_n}\left( 1-g_{n-1}^\mathbf{b} (j-2) + \sum_{r=0}^{n-1} (-1)^r \binom{n+1}{r+2} g_{n-1}(r,j-2)  \right)  & \mbox{if }j\geq 2.
	\end{cases}
	\end{equation}
\end{thm}

These two recurrences determine $K^{-1}_n((i,i),(j,j))$ for $0 \leq i,j\leq 2n-1$, and along with the equations $K_n^{-1}.K_n=K_n.K_n^{-1}= \mathbbm{I}$ viewed entrywise fully determines the entries of $K_n^{-1}$.  

To prove these results, we need two local moves for dimers.

\begin{figure}[htbp!]
	\begin{center}
		\includegraphics[height=4cm]{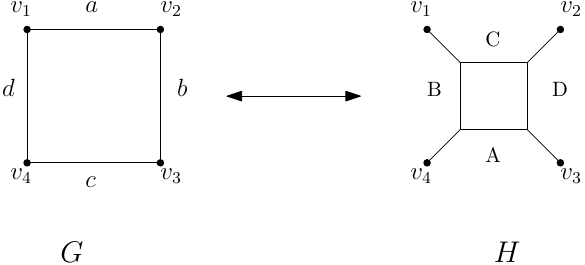}
		\caption{The spider move.} 
		\label{fig:spider}
	\end{center}
\end{figure}

\begin{enumerate}
	\item \emph{Spider move}: Suppose we have a large square with edge weights $a,b,c$ and $d$ (clockwise labelling) on some graph $G$.  This square can be deformed to smaller square with additional edges added between the vertices of the original square and the vertices of the smaller square as shown in \cref{fig:spider} to form a new graph called $H$. If the new edge weights in $H$ around the smaller square, $A,B,C$ and $D$ are related to the old weights in $G$ by $A=a/\Delta$, $B=b/\Delta$, $C=c/\Delta$ and $D=d/\Delta$, where $\Delta = a b +c d$
and the edge weights for the additional edges is 1, then local configurations and weights of matchings are preserved under the transformation from $G$ to $H$. This transformation is called the {\em spider move}~\cite{Pro03}, and we have $Z_G=\Delta Z_H$.

\begin{figure}[htbp!]
	\begin{center}
		\includegraphics[height=1cm]{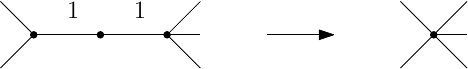}
		\caption{The edge contraction move.} 
		\label{fig:edgecon}
	\end{center}
\end{figure}

	\item \emph{Edge Contraction}: If a vertex is incident to two edges each having weight 1, contract the two incident edges. This does not change the partition function; see \cref{fig:edgecon}.
	
\end{enumerate}

\begin{proof}[Proof of \cref{thm:boundaryrec2}]
	For the purpose of the proof, write $\mathbf{i}$ to be the vertex $(i,i)$ for $1 \leq i \leq 2n-1$ and let $\mathbf{a}$ denote the vertex $(0,2n+1)$.  

	We start with the proof of~\eqref{eq:gnij}.  The first two conditions are immediate by definition. For $0\leq i<j <2n-1$, we apply graphical condensation, similar to~\eqref{eq:Kuo} in the proof of \cref{lem:condensation}, which gives
	\begin{equation}\label{eq:Kuobottom}
	Z_n^{\{\mathbf{a},\mathbf{i},\mathbf{j},\mathbf{b}\}}Z_n = Z_n^{\{\mathbf{a},\mathbf{i}\}}Z_n^{\{\mathbf{j},\mathbf{b}\}}-Z_n^{\{\mathbf{a},\mathbf{j}\}}Z_n^{\{\mathbf{i},\mathbf{b}\}}+Z_n^{\{\mathbf{a},\mathbf{b}\}}Z_n^{\{\mathbf{i},\mathbf{j}\}}.
	\end{equation}
We first consider the above equation when $i \geq 2$. 	
	As in the proof of \cref{lem:condensation}, removing $\mathbf{a}$ freezes off edges; see \cref{fig:Kuo}. This means that $Z_n^{\{\mathbf{a},\mathbf{b}\}}=Z_{n-1}$, $Z_n^{\{\mathbf{a},\mathbf{i\}}}=Z_{n-1}^{\{\mathbf{i-2},\mathbf{b}\}}$ and $Z_n^{\{\mathbf{a},\mathbf{i},\mathbf{j},\mathbf{b}\}}= Z_{n-1}^{\{\mathbf{a},\mathbf{i},\mathbf{j},\mathbf{b}\}}.$  Dividing both sides of~\eqref{eq:Kuobottom} by $Z_{n-1}Z_{n}$ gives the last condition in~\eqref{eq:gnij}.  Next, when $i=0$, notice that removing both $\mathbf{a}$ and $\mathbf{0}$ has no impact on the forced edges, which means that $Z_n^{\{\mathbf{a},\mathbf{0},\mathbf{j},\mathbf{b}\}}=Z_{n-1}^{\{\mathbf{j-2},\mathbf{b}\}}$. Dividing~\eqref{eq:Kuobottom} by $Z_{n-1}Z_n$ in this case gives
	\begin{equation}
		g_{n-1}(0,j-2)=g_n^\mathbf{b}(j)-g_{n-1}^\mathbf{b}(j-2)+g_n(0,b)
	\end{equation}
	and so 
	\begin{equation}\label{eq:diff}
	g_{n-1}(0,j-2)+g_{n-1}^\mathbf{b}(j-2)=g_n^\mathbf{b}(j)+g_n(0,j).
	\end{equation}
	The equality in~\eqref{eq:diff} is equal to 1 when $n$ is even which follows immediately from the last equation in \cref{lem:immediateg}. We now consider~\eqref{eq:diff} when $n$ is odd.  To do so, set $i=0$ and $j=1$ in~\eqref{eq:Kuobottom} to get 
	\begin{equation}
Z_n^{\{\mathbf{a},\mathbf{0},\mathbf{1},\mathbf{b}\}}Z_n = Z_n^{\{\mathbf{a},\mathbf{0}\}}Z_n^{\{\mathbf{1},\mathbf{b}\}}-Z_n^{\{\mathbf{a},\mathbf{1}\}}Z_n^{\{\mathbf{0},\mathbf{b}\}}+Z_n^{\{\mathbf{a},\mathbf{b}\}}Z_n^{\{\mathbf{0},\mathbf{1}\}}.
	\end{equation}
	Notice that due to edges being forced $Z_n^{\{\mathbf{a},\mathbf{0},\mathbf{1},\mathbf{b}\}}=Z_{n-1}$ and $Z_n^{\{\mathbf{a},\mathbf{1}\}}=0$, where the latter follows since the dimer covering of $\mathtt{V}_n \backslash\{\mathbf{a},(1,1)\}$  is zero as the	induced dimers from removing $\mathbf{a}$ from the graph are incompatible with removing $(1,1)$ from the graph; compare ~\cref{fig:Kuo} and~\cref{fig:forced1}. 
		Dividing the above equation by $Z_{n-1}Z_n$ gives
	\begin{equation}
		1=g_n^{\mathbf{b}}(1)+g_n(0,1).
	\end{equation}
	We have shown~\eqref{eq:diff} is equal to 1 when $n$ is odd, which verifies the third condition in~\eqref{eq:gnij}. Finally, when $i=1$ in~\eqref{eq:Kuobottom}, gives
\begin{equation}
Z_n^{\{\mathbf{a},\mathbf{1},\mathbf{j},\mathbf{b}\}}Z_n = Z_n^{\{\mathbf{a},\mathbf{1}\}}Z_n^{\{\mathbf{j},\mathbf{b}\}}-Z_n^{\{\mathbf{a},\mathbf{j}\}}Z_n^{\{\mathbf{1},\mathbf{b}\}}+Z_n^{\{\mathbf{a},\mathbf{b}\}}Z_n^{\{\mathbf{1},\mathbf{j}\}}.
	\end{equation}
	Due to their being no dimer covering of $\mathtt{V}_n \backslash\{\mathbf{a},(1,1)\}$  as mentioned above. The above equation becomes 
	\begin{equation}
0=-Z_n^{\{\mathbf{a},\mathbf{j}\}}Z_n^{\{\mathbf{1},\mathbf{b}\}}+Z_n^{\{\mathbf{a},\mathbf{b}\}}Z_n^{\{\mathbf{1},\mathbf{j}\}}.
	\end{equation}
	Using that $Z_n^{\{\mathbf{a},\mathbf{j}\}}=Z_{n-1}^{\{ \mathbf{j}, \mathbf{b}\}}$ and dividing by $Z_{n-1}Z_n$ and using \cref{lem:immediateg} gives the fourth condition in~\eqref{eq:gnij}.
\end{proof}

	\begin{figure}[htbp!]
		\begin{center}
		\begin{tabular}{cc}
			\includegraphics[height=6cm]{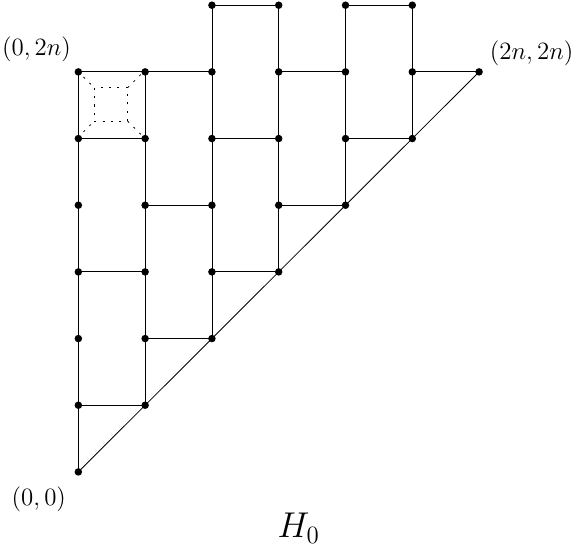} &
			\includegraphics[height=6cm]{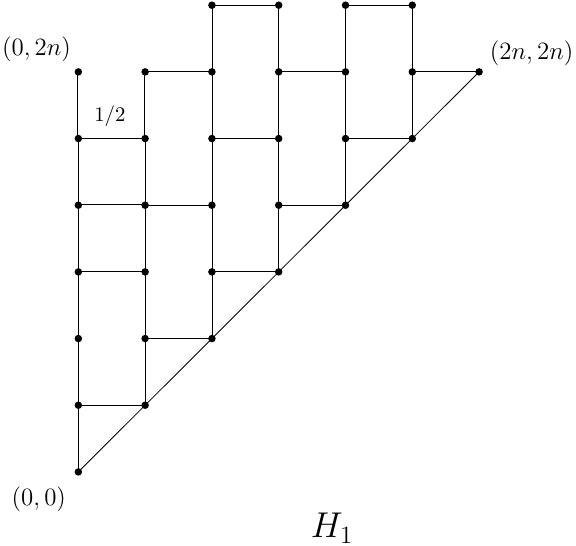}		
		\end{tabular}
			\caption{The left graph $H_0$ is obtained after applying an edge contraction. Now, a square appears where we can apply the spider move. After applying the spider move and contracting edges, we obtain the right graph $H_1$.} 
			\label{fig:squarepush}
		\end{center}
	\end{figure}

\begin{proof}[Proof of \cref{thm:boundaryrec1}]
The first two conditions follow from the third and first conditions in \cref{lem:immediateg}. The last condition in~\eqref{eq:gbj} is more involved and we illustrate the steps to find a recurrence of the partition function first.  
	
We perform an edge contraction on the edges $((0,2n),(0,2n+1))$ and $((0,2n+1),(1,2n+1))$. This means that there is an edge between $(1,2n)$ and $(0,2n)$ as well as a square face with coordinates $(0,2n-1),(1,2n-1),(1,2n)$ and $(0,2n)$; see the left graph in \cref{fig:squarepush}. Label this graph $H_0$.  
To the square above, we apply the spider move and edge contraction on the bottom two edges protruding from the new (smaller) square; see the right figure in \cref{fig:squarepush}. 
	These operations have
	\begin{enumerate}
		\item	deleted the edge $((0,2n),(1,2n))$ on $H_0$, 
	\item added an edge between $((0,2n-2), (1,2n-2))$ on $H_0$,
	\item changed the edge weights on the edges $((0,2n-2),(1,2n-2))$, $((1,2n-2),(1,2n-1))$, $((0,2n-1),(1,2n-1))$ and $((1,2n-2),(1,2n-1))$ to 1/2 on $H_0$. 
	\end{enumerate}
Label this new graph $H_1$. This operation gives 
	\begin{equation}\label{eq:sqmvonce}
	Z_{H_0}=2 Z_{H_1}. 
	\end{equation}

		We now proceed iteratively and describe the step from $H_{k-1}$ to $H_{k}$ from applying the square move for $2 \leq k \leq n-1$.  On the graph $H_{k-1}$, we apply the spider move on the square face whose center is given by $(1/2,2n-2k+3/2)$ and applying edge contraction on the two bottom edges protruding from the new (smaller) square; see \cref{fig:spiderdown}. 

\begin{figure}[htbp!]
		\begin{center}
			\includegraphics[height=6cm]{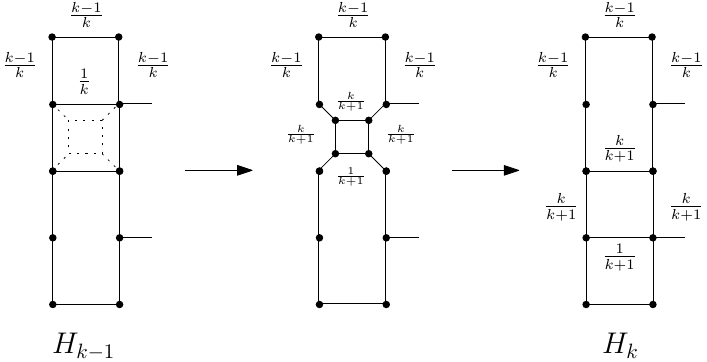}
			\caption{The local moves taking $H_{k-1}$ to $H_k$. In each of the three graphs, the top left vertex is $(0,2n-2k+3)$.
All unlabelled edges have weight $1$.} 
			\label{fig:spiderdown}
		\end{center}
\end{figure}

		We call this new graph $H_{k}$. These operations have
\begin{enumerate}
	\item	deleted the edge $((0,2n-2k+2),(1,2n-2k+2))$ on $H_{k-1}$, 
	\item added an edge between $((0,2n-2k), (1,2n-2k))$ on $H_{k-1}$,
	\item changed the edge weights of the edges, $((1,2n-2k),(1,2n+1-2k))$, $((0,2n-2k+1),(1,2n-2k+1))$ and $((1,2n-2k),(1,2n-2k+1))$ to $k/(k+1)$ and changed the edge weight of the edge $((0,2n-2k),(1,2n-2k))$ to $1/(k+1)$ in $H_{k-1}$.
	\end{enumerate}
	Label this new graph $H_k$. This operation gives 
	\begin{equation}\label{eq:sqmvktimes}
		Z_{H_{k-1}}=\frac{k+1}{k} Z_{H_k}.
	\end{equation}

	From~\eqref{eq:sqmvonce} and~\eqref{eq:sqmvktimes} we have
	\begin{equation} \label{eq:aftersquares}
		Z_n = n Z_{H_{n-1}}
	\end{equation}
	From the above operations, the edges $((1,2n-2k),(1,2n-2k+1))$ have weight $\frac{k}{k+1}$ for $1 \leq k \leq n-1$. 
	Since the graph $H_{n-1}$ contains a pendant edge $((0,2n-1),(0,2n))$, this can be removed inducing another pendant edge. Iteratively removing these pendant edges $((0,2n-1-2k),(0,2n-2k))$ for $ 1 \leq k \leq n-2$ from $H_{n-1}$ leaves us with the graph in \cref{fig:finalmove}.

\begin{figure}[htbp!]
		\begin{center}
			\includegraphics[height=6cm]{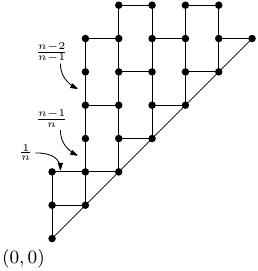}
\caption{The situation after removing pendant edges from  $H_{n-1}$. All unlabelled edges have weight $1$.} 
			\label{fig:finalmove}
		\end{center}
\end{figure}

For the vertex $\mathbf{0} = (0,0)$, either there is a dimer on the edge $(\mathbf{0},(0,1))$ or there is a dimer on the edge $(\mathbf{0},(1,1))$.  For the former, this forces a dimer on the edge $((0,2),(1,2))$ (which has weight $1/n$) and forces dimers on the edges $((1,1),(2,2))$ and $((1,2k+1),(1,2k+2))$ for $1\leq k \leq n-1$, leaving us with the graph with vertices $V_{n-1} \backslash \{\mathbf{0},\mathbf{b}\}$. It follows from \cref{lem:immediateg} that the number of dimer covers on this graph is  $Z_{n-1}$. For the latter, this induces a dimer on the edge $((0,1),(0,2)$.  In this case, only one of the edges $((1,2+2k),(2,2+2k))$ for $0 \leq k \leq n-1$ can be covered by a dimer and since the edges $((1,2n-2k),(1,2n-2k+1))$ have weight $\frac{k}{k+1}$ for $1 \leq k \leq n-1$, we find using \eqref{eq:aftersquares} that 
	\begin{equation} \label{eq:partitionrecurrence}
		Z_n =  Z_{n-1} + \sum_{k=0}^{n-1} (n-k) Z_{n-1}^{\{\mathbf{0},(0,2k+1)\}}.
	\end{equation}

We can now prove the third equation in~\eqref{eq:gbj}.
To so, we consider the same steps that lead to \eqref{eq:partitionrecurrence} but replace the graph of $\mathtt{V}_n$ by the induced graph of $\mathtt{V}_n$ with the vertices $(j,j)$ and $\mathbf{b}$ for $j \geq 2$ removed.  We then arrive at 
\begin{equation} \label{eq:partitionrecurrence2}
	Z_n^{\{\mathbf{j},\mathbf{b}\}} =  Z_{n-1}^{\{\mathbf{j-2},\mathbf{0}\}} + \sum_{k=0}^{n-1} (n-k) Z_{n-1}^{\{(0,2k),\mathbf{j}\}}.
\end{equation}
Dividing the above equation by $Z_n$ and using the first condition in~\eqref{eq:gnij} gives
\begin{equation}
	g_n^{\mathbf{b}} (j)= \frac{Z_{n-1}}{Z_n} \left( 1- g_{n-1}^{\mathbf{b}}(j-2)\right)  + \frac{Z_{n-1}}{Z_n}\sum_{k=0}^{n-1} (n-k) \frac{Z_{n-1}^{\{(0,2k),\mathbf{j-2}\}}}{Z_{n-1}}.
\end{equation}
Since the vertex $(0,2k)$ for $0\leq k \leq n-1$ is on the leftmost boundary, this can be moved to the boundary $\{(j,j): 0 \leq j \leq 2n-1\}$ by applying the matrix equation $K_n.K_n^{-1} = \mathbbm{I}$ entrywise. This gives 
\begin{equation}
	g_n^{\mathbf{b}} (j)= \frac{Z_{n-1}}{Z_n} \left( 1- g_{n-1}^{\mathbf{b}}(j-2)\right)  + \frac{Z_{n-1}}{Z_n}\sum_{k=0}^{n-1} (n-k)\sum_{r=0}^k \binom{k}{r} (-1)^r \frac{Z_{n-1}^{\{\mathbf{r},\mathbf{j-2}\}}}{Z_{n-1}}.
\end{equation}
Rearranging the sums and evaluating the $k$-sum gives the final equation in~\eqref{eq:gbj}.
\end{proof} 

We now prove the sum rule.

\begin{proof}[Proof of \cref{thm:gnb formula}]
Substituting \eqref{Knb0} and \eqref{hnb0} into~\eqref{eq:Kinvboundarycount}, we have that 
\begin{equation}
\begin{split}
\sum_{j=0}^{2n} g_n^{\mathbf{b}}(j)&=\sum_{j=0}^{2n} (-1)^{j+1} K_n^{-1}((j,j),\mathbf{b})=\sum_{j=0}^{2n} (-1)^{j+1}h_n^{0,
\mathbf{b}}(j)\\
&= \sum_{j=0}^{2n}-(-1)^{j+1} [j+1]_2 +\sum_{j=0}^{2n}(-1)^{j+1} \sum_{k=0}^n p(n,k,0) \frac{1}{2\pi \mathrm{i}} \int_{
\Gamma_0} \text{d}r \frac{ (1+r)^{n-k}}{(1-r) r^{j-2k}}.
\end{split}
\end{equation}
Both sums in $j$ can be evaluated on the right side of the above equation giving
\begin{equation}
\sum_{j=0}^{2n} g_n^{\mathbf{b}}(j)=n+1- \sum_{k=0}^n p(n,k,0) \frac{1}{2\pi \mathrm{i}} \int_{
\Gamma_0} \text{d}r \frac{ (1+r)^{n-k}}{(1-r)(1+r) r^{2n-2k}}.
\end{equation}
where we have only kept the term with a residue contribution at $r=0$ when evaluating the geometric sum. Notice that the second term on the right side of the above equation does not have a residue contribution at $r=0$ when $k=n$ and so we have
\begin{equation}
\begin{split}
\sum_{j=0}^{2n} g_n^{\mathbf{b}}(j)&=n+1- \sum_{k=0}^{n-1} p(n,k,0) \frac{1}{2\pi \mathrm{i}} \int_{
\Gamma_0} \text{d}r \frac{ (1+r)^{n-k-1}}{(1-r)r^{2n-2k}}\\
&=n+1-\frac{1}{2} \sum_{k=0}^{n-1} p(n,k,0) 2^{n-k},
\end{split}
\end{equation}
where the last line follows from pushing the contour through $\infty$ picking up the residue at $r=1$.  The result then follows from \cref{thm:sumf} and from the fact that $p(n,n,0)=(-1)^n$.
\end{proof}

\section{Heuristics for the limit Shape} 
\label{sec:conjlimitshape}

To obtain the conjectured limit shape formula, we only consider the asymptotics of $K_n^{-1}((x_1,x_2),\mathbf{b})$ for $x_2 \in 2\mathbb{Z}+1$ for $(x_1,x_2)$ rescaled as given in \cref{conj:limitshape}.  Strictly speaking, this term does not contain any probabilistic information but we expect to see a similar structure when analyzing other entries of the inverse Kasteleyn matrix when both terms are close to the limit shape curves.  
To find these asymptotics, we express the formula for $K^{-1}_n((x_1,x_2),\mathbf{b})$ as a single contour integral and apply the method of steepest descent.  This will give a function whose double roots parameterize the limit shape curves. This is a fairly standard approach in the asymptotics of random tilings, see for instance~\cite[Lecture 15]{Gor20} for a good exposition, and so we only give a brief outline of the main steps. Even though we can give the computations below in full detail, we cannot analyze the rest of the entries of the inverse Kasteleyn matrix with this method, as they are currently not in the best form for asymptotic analysis.

Since we have $(x_1,x_2)=(i_1,i_1+2i_2+1)$, we use the expression gathered from~\eqref{Knb1} and~\eqref{hnb1}. Note, we can truncate the $k$-sum to $\lfloor \frac{i_1}{2} \rfloor$ since there are no residues at $r=0$ for $k>\lfloor \frac{i_1}{2} \rfloor$.  We obtain
 \begin{equation}\label{Kinvbounint}
	\begin{split}
	K^{-1}_n((x_1,x_2),\mathbf{b}) &= (-1)^{i_2} \sum_{k=0}^{\lfloor \frac{i_1}{2} \rfloor} \frac{p(n,k,0)}{2 \pi \mathrm{i}} \int_{\Gamma_0} \frac{\text{d}r}{r} \frac{(1+r)^{n-k}}{r^{i_1-2k}} \sum_{\ell=0}^{i_2} \binom{i_2-1+\ell}{\ell} (-r)^\ell  \\
		&= (-1)^{i_2} \sum_{k=0}^{\lfloor \frac{i_1}{2} \rfloor}   \frac{p(n,k,0)}{2 \pi \mathrm{i}} \int_{\Gamma_0} \frac{\text{d}r}{r} \frac{ (1+r)^{n-k-i_2}}{r^{i_1-2k}},
	\end{split}
\end{equation}
where we have used 
\begin{equation}
	\sum_{\ell=0}^{i_2} \binom{i_2-1+\ell}{\ell} (-r)^\ell = (1+r)^{-i_2} -(-r)^{i_1+1}  \binom{i_1+i_2}{1+i_1}\pFq{2}{1}{1,1+i_1+i_2}{2+i_2}{-r},
\end{equation}
and the fact that the latter term will have no residue at $r=0$ in the middle equation in~\eqref{Kinvbounint}.  Computing the integral in the last equation in~\eqref{Kinvbounint}, using~\eqref{eq:pnkl}  and writing factorials as gamma functions gives
\begin{equation} \label{Kinvbounint2}
	\begin{split}
		&(-1)^{i_2}\sum_{k=0}^{\lfloor \frac{i_1}{2} \rfloor} \frac{(-1)^k}{k! (n-k)!} \frac{\Gamma(n+k+2) \Gamma(2n-k+2) \Gamma(2n-2k+1) }{\Gamma(3n-k+3)\Gamma(n-k+2)} \\
		&\times (3n-3k+2)\frac{\Gamma(n-k+i_2+1)}{\Gamma(i_1-2k+1)\Gamma(n+k-i_1-i_2+1) }.
	\end{split}
\end{equation}
Using
\begin{equation}
	\frac{(-1)^k}{k! (n-k)!} =\int_{\Gamma_k} \text{d}w \frac{(-1)^n}{\prod_{r=0}^n (w-r)}, \quad 0 \leq k \leq n,
\end{equation}
and the residue theorem, we have that~\eqref{Kinvbounint2} is equal to 
\begin{equation}
	\begin{split}\label{Kinvbounint3}
\frac{(-1)^{i_2}}{2 \pi \mathrm{i}}& \int_{\Gamma_{0,\dots,\lfloor i_1/2 \rfloor}} \text{d}w  \frac{ (-1)^n}{\prod_{r=0}^n (w-r)} \frac{\Gamma(n+w+2) \Gamma(2n-w+2) \Gamma(2n-2w+1)}{\Gamma(3n-w+3)\Gamma(n-w+2) } \\
		&\times  \frac{(3n-3w+2)\Gamma(n-w+i_2+1)}{\Gamma(i_1-2w+1)\Gamma(n+w-i_1-i_2+1) }.
	\end{split}\end{equation}

\begin{figure}[htbp!]
		\begin{center}
			\includegraphics[height=5cm]{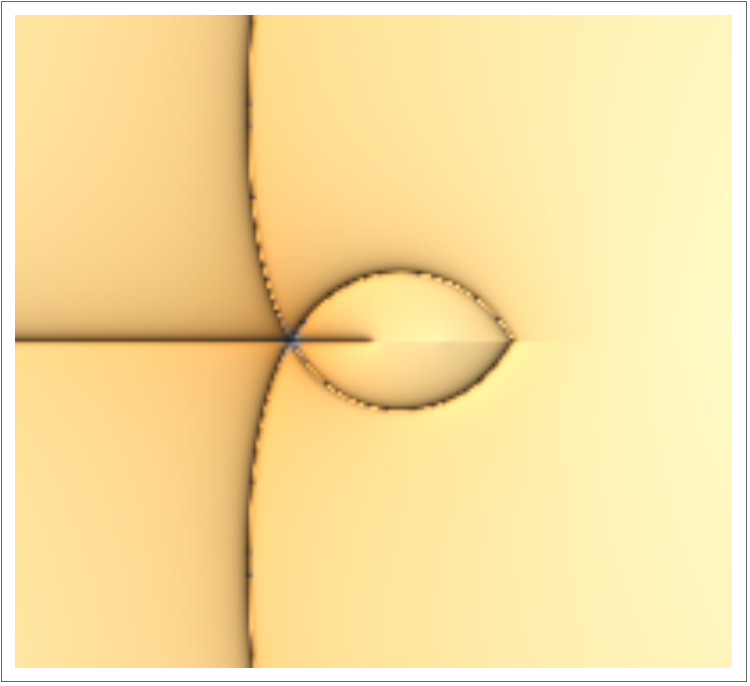}
			\caption{The contours of steepest descent and ascent for $s(w)$. The contour of steepest descent leaves the double critical point with angle $+\pi/3$ and passes through the point $(X+2)>0$ and then returns to the double critical point. The crossing points of the real axis can be determined explicitly.}
			\label{fig:contours}
		\end{center}
\end{figure}	
	
Using the rescaling $x_1=[(X+2)n]$ and $x_2=[(\sqrt{3}Y+2)n]$
given in the statement of \cref{conj:limitshape} gives that $i_1=[(X+2)n]$ and $i_2=\frac{1}{2} ([(\sqrt{3}Y+2)n]-[(X+2)n]-1)$. In~\eqref{Kinvbounint3}, we make the change of variables $w \mapsto wn$ and apply Stirling's formula, see~\cite[Proposition 7.3]{Pet14} for the exact form. Then we obtain
\begin{equation}
	\frac{(-1)^{i_2+n}}{2\pi \mathrm{i}} \int_{\tilde{\Gamma}} \text{d}w \; h(w)e^{ns(w)+O(1/n)}
\end{equation}
	where $\tilde{\Gamma}$ is a positively oriented contour that surrounds the points $0,\frac{1}{n},\dots, \lfloor \frac{i_1}{2} \rfloor \frac{1}{n}$, $h(w)$ is a rational function in $w$ and 
	\begin{equation}
		\begin{split}
			&s(w)= (1+w)\log (1+w) +(2-w) \log (2-w) +(2-2w)\log (2-2w) \\
			&+ \left(1-w-\frac{\sqrt{3}Y-X}{2} \right) \log  \left(1-w-\frac{\sqrt{3}Y-X}{2} \right) -w \log (-w)  \\
			&-(3-w)\log(3-w) -2(1-w) \log (1-w) -(X+2-2w) \log (X+2-2w) \\
			&-\left(1+w-(2+X)-\frac{\sqrt{3}Y-X}{2} \right) \log \left(1+w-(2+X)-\frac{\sqrt{3}Y-X}{2} \right).
	\end{split}
	\end{equation}
	The exact form of $h(w)$ is not important (in fact, it can be computed explicitly). What is important is that it does not influence the saddle point function $s(w)$, nor does it contain any additional poles when we deform the contours.  The roots of $s(w)$ can be determined by solving $s'(w)=0$ in $w$ which gives
	\begin{equation}
\frac{2 X-Y^2+4\pm \sqrt{Y^2 \left(X^2+Y^2-4\right)}}{4-Y^2}.
	\end{equation}

	This has double roots when $Y=0$, which corresponds to the top boundary of the rescaled TSSCPP, and when $X^2+Y^2=4$, which is precisely the conjectured limit shape curve. We focus on the latter and set $Y=-\sqrt{4-X^2}$.  Due to the rescaling of the TSSCPP, we have that $-2 < X  < -\sqrt{3}$.  The contour $\tilde{\Gamma}$ can be deformed to pass through the double root, following the contours of steepest descent; see \cref{fig:contours} for a description. No other poles are crossed when performing this deformation and so the main contribution comes locally around the double critical point.  This leads to Airy function type asymptotics, which can easily be computed; see~\cite{Joh17} for an example. 
We omit the details and the explicit computation since we have already evaluated the limit shape curves.  

\section*{Acknowledgements}
The first author (SC) would like to thank Dan Romik for sharing his slides from his talk.
The second author (AA) would like to acknowledge insightful discussions with C. Krattenthaler on \cref{thm:sumf}.
We also thank the anonymous reviewers for their comments.
This material is based upon work supported by the Swedish Research Council under grant no. 2016-06596 while the authors were in residence at Institut Mittag-Leffler in Djursholm, Sweden during the spring semester of 2020.
This research was supported in part by the International Centre for Theoretical Sciences (ICTS) during a visit for participating in the program - Universality in random structures: Interfaces, Matrices, Sandpiles (Code: ICTS/urs2019/01).
The authors acknowledge support from the Royal Society grant  IES\textbackslash R1\textbackslash 191139.
The first author (AA) was partially supported by the UGC Centre for Advanced Studies and by Department of Science and Technology grant
EMR/2016/006624. 
The second author (SC) was supported by EPSRC EP\textbackslash T004290\textbackslash 1. 

\bibliographystyle{alpha}
\bibliography{Biblio}

\end{document}